\documentclass[11pt,a4paper]{amsart}
\usepackage{amsmath,amssymb,amsthm,amsfonts}
\usepackage{mathrsfs}
\usepackage{amscd}
\usepackage{listings}
\usepackage{paralist}
\usepackage{booktabs}
\usepackage[usenames,dvipsnames]{color}
\usepackage{comment}
\usepackage{graphicx}
\usepackage{latexsym}
\usepackage[shortlabels]{enumitem}
\usepackage{thmtools}
\usepackage{url}
\usepackage[all,cmtip]{xy}
\usepackage[bookmarks=false]{hyperref} 

\newtheorem{theorem}{Theorem}[section]

\newtheorem{remark}[theorem]{Remark}
\newtheorem{cor}[theorem]{Corollary}
\newenvironment{corollary}{\begin{cor} \em}{\end{cor}}
\newtheorem{tont}[theorem]{Definition}
\newenvironment{definition}{\begin{tont} \em}{\end{tont}}
\newtheorem{lemma}[theorem]{Lemma}

\newtheorem{proposition}[theorem]{Proposition}
 \DeclareMathOperator\codim{codim}
\DeclareMathOperator\Ker{Ker}
\DeclareMathOperator\Ima{Im}
\begin{document}

\title [Associated points and integral closure of modules]{Associated points and integral closure of modules}
\author{Antoni Rangachev}
\begin{abstract} Let $X:=\mathrm{Spec}(R)$ be an affine Noetherian scheme, and $\mathcal{M} \subset \mathcal{N}$ be a pair of finitely generated $R$-modules. Denote their Rees algebras by $\mathcal{R}(\mathcal{M})$ and $\mathcal{R}(\mathcal{N})$. Let  $\mathcal{N}^{n}$ be the $n$th homogeneous component of  $\mathcal{R}(\mathcal{N})$ and let $\mathcal{M}^{n}$ be the image  of the $n$th homegeneous component of $\mathcal{R}(\mathcal{M})$ in $\mathcal{N}^n$. Denote by $\overline{\mathcal{M}^{n}}$ be the integral closure of $\mathcal{M}^{n}$ in $\mathcal{N}^{n}$. We prove that  $\mathrm{Ass}_{X}(\mathcal{N}^{n}/\overline{\mathcal{M}^{n}})$ and  $\mathrm{Ass}_{X}(\mathcal{N}^{n}/\mathcal{M}^{n})$ are asymptotically stable, generalizing known results for the case where $\mathcal{M}$ is an ideal or where $\mathcal{N}$ is a free module. Suppose either that $\mathcal{M}$ and $\mathcal{N}$ are free at the generic point of each irreducible component of $X$ or $\mathcal{N}$ is contained in a free $R$-module. When $X$ is universally catenary, we prove a  generalization of a classical result due to McAdam and obtain a geometric classification of the points appearing in $\mathrm{Ass}_{X}(\mathcal{N}^{n}/\overline{\mathcal{M}^{n}})$. Notably, we show that if $x \in \mathrm{Ass}_{X}(\mathcal{N}^{n}/\overline{\mathcal{M}^{n}})$ for some $n$, then $x$ is the generic point of a codimension-one component of the nonfree locus of $\mathcal{N}/\mathcal{M}$ or $x$ is a generic point of an irreducible set in $X$ where the fiber dimension $\mathrm{Proj}(\mathcal{R}(\mathcal{M})) \rightarrow X$ jumps. We prove a converse to this result without requiring $X$ to be universally catenary.  Our approach is geometric in spirit. Also, we recover, strengthen, and prove a sort of converse of an important result of Kleiman and Thorup about integral dependence of modules.
\end{abstract}
\address{Department of Mathematics\\
  Northeastern University\\
  Boston, MA 02215\\
  Institute of Mathematics \\
and Informatics, Bulgarian Academy of Sciences\\
Akad. G. Bonchev, Sofia 1113, Bulgaria}
\maketitle
\tableofcontents
\section{Introduction} Let $X:=\mathrm{Spec}(R)$ be an affine Noetherian scheme and let $\mathcal{M}$ be a finitely generated $R$-module. Following Eisenbud, Huneke and Ulrich \cite{Eisenbud}, define the {\it Rees algebra} of $\mathcal{M}$ as the quotient $$\mathcal{R}(\mathcal{M}):=\mathrm{Sym}(\mathcal{M})/(\cap \mathcal{L}_g)$$
where the intersection is taken over all homomorphisms $g$ from $\mathcal{M}$ to a free $R$-module $\mathcal{F}_g$ and $\mathcal{L}_g$ denotes the kernel of the induced map $\mathrm{Sym}(\mathcal{M}) \rightarrow \mathrm{Sym}(\mathcal{F}_g)$. In fact as shown in Prp.\ 1.3 \cite{Eisenbud} $\mathcal{R}(\mathcal{M})$ can be computed from any homomorphism of $\mathcal{M}$ to a free $R$-module whose dual is surjective.

Let $\mathcal{M} \subset \mathcal{N}$ be a pair of finitely generated $R$-modules. Assume that $\mathcal{M}$ and $\mathcal{N}$ are either free at the generic point of each irreducible component of $X$, or they are contained in a free $R$-module. The inclusion of $\mathcal{M}$ into $\mathcal{N}$ induces a map from $\mathcal{R}(\mathcal{M})$ to $\mathcal{R}(\mathcal{N})$.
Denote by $\mathcal{N}^{n}$ the $n$th homogeneous  component of $\mathcal{R}(\mathcal{N})$ and by $\mathcal{M}^{n}$ the image of the $n$th homogeneous component of $\mathcal{R}(\mathcal{M})$ in $\mathcal{R}(\mathcal{N})$. Finally, recall that the {\it integral closure} $\overline{\mathcal{M}^{n}}$ of $\mathcal{M}^{n}$ in $\mathcal{N}^{n}$ is the module generated by those elements from $\mathcal{N}^{n}$ that satisfy an equation in $\mathcal{R}(\mathcal{N})$ of integral dependence over $\mathcal{R}(\mathcal{M})$.

Our main result, Thm.\ \ref{conormal}, is as follows. Assume $X$ is local and universally catenary with closed point $x_0$. Assume the codimension of each irreducible component of the fiber of  $\mathrm{Proj}(\mathcal{R}(\mathcal{M}))$ over $x_0$ is at least $2$. Let $h$ be a general element of the maximal ideal of $\mathcal{O}_{X,x_0}$. Then $h \notin \mathrm{z.div}(\mathcal{N}^n/\overline{\mathcal{M}^n})$ for each $n$. Furthermore, we describe explicitly the genericity conditions on $h$. In Scts.\ 5-7 we apply Thm. \ref{conormal} to two seemingly unrelated problems from commutative algebra and algebraic geometry.

First we analyze the set $\bigcup_{n=1}^{\infty}\mathrm{Ass}_{X}(\mathcal{N}^{n}/\overline{\mathcal{M}^{n}})$. It is a hard problem to show that this set is finite, because the modules $\overline{\mathcal{M}^{n}}$ may not form a finitely generated $R$-algebra. Thm.\ \ref{conormal} yields a complete classification of the points $x$ of $X$  that appear in this set when $X$ is universally catenary. They are of two types: generic points of codimension one components of the nonfree locus of $\mathcal{N}/\mathcal{M}$, and generic points of the irreducible components of closed subsets of $X$ where the fiber dimension of the structure morphism $\mathrm{Proj}(\mathcal{R}(\mathcal{M})) \rightarrow X$ jumps. Then Chevalley's constructibility result implies that there are finitely many such $x$ and hence $\bigcup_{n=1}^{\infty}\mathrm{Ass}_{X}(\mathcal{N}^{n}/\overline{\mathcal{M}^{n}})$ is a finite set. 

More generally, without assuming $X$ to be universally catenary and any additional hypothesis on the pair $\mathcal{M} \subset \mathcal{N}$, we prove algebraically that $\mathrm{Ass}_{X}(\mathcal{N}^{n}/\overline{\mathcal{M}^{n}}) \subseteq \mathrm{Ass}_{X}(\mathcal{N}^{n+1}/\overline{\mathcal{M}^{n+1}})$ for each $n$, and that the sets $\mathrm{Ass}_{X}(\mathcal{N}^{n}/\overline{\mathcal{M}^{n}})$ and $\mathrm{Ass}_{X}(\mathcal{N}^{n}/\mathcal{M}^{n})$  are asymptotically stable as special cases of  more general results about finitely generated $R$-algebras. Working in complete generality though, doesn't yield a satisfactory geometric classification of the points appearing in $\mathrm{Ass}_{X}(\mathcal{N}^{n}/\overline{\mathcal{M}^{n}})$ as in the case when $X$ is universally catenary.

The analogous problems for ideals $\mathcal{I}$ were posed by Rees \cite{Rees56}. First, Ratliff \cite{Ratliff} proved that $\bigcup_{i=1}^{\infty}\mathrm{Ass}_{X}(R/\mathcal{I}^{n})$ is finite. Then Brodmann \cite{Brodmann} inspired by Ratliff's work, showed that $\mathrm{Ass}_{X}(R/\mathcal{I}^{n})$ is asymptotically stable. 
If $\mathrm{ht}(\mathcal{I}) \geq 1$, it follows from Ratliff's Thm.\ 2.5 \cite{Ratliff} and McAdam and Eakin's Prp.\ 7 \cite{Eakin} that $\mathrm{Ass}_{X}(R/\overline{\mathcal{I}^{n}})$ is asymptotically stable. Then Rees \cite{Rees81} observed that 
$\bigcup_{i=1}^{\infty}\mathrm{Ass}_{X}(R/\overline{\mathcal{I}^{n}})$ is finite in complete generality as a consequence of his valuation theorem. Finally, Ratliff \cite{Ratliff84} proved that $\mathrm{Ass}_{X}(R/\overline{\mathcal{I}^{n}}) \subseteq \mathrm{Ass}_{X}(R/\overline{\mathcal{I}^{n+1}})$ for each $n$ without the assumption $\mathrm{ht}(\mathcal{I}) \geq 1$. These results were extended by Katz and Naude \cite{Katz1} to the case where $\mathcal{M}$ is contained in a free $R$-module $\mathcal{F}$ by reduction to the ideal case.

Questions about the asymptotic behavior of prime divisors arise naturally in various problems from commutative algebra and algebraic geometry, including going-down of prime ideals, catenary chain conjectures (see \cite{McAdam2} and \cite{Ratliff83}), symbolic powers of ideals, Swanson's result on the asymptotics of the primary decomposition of the powers of an ideal \cite{Swanson} and its applications to symbolic powers (cf. \cite{ELS}). More recently, the author \cite{Rangachev} has found applications to a specialization problem for certain local cohomology modules defined over flat families, and to characterizing the vanishing of certain restricted local volumes that are invariants of deformations. 

Assume $X$ is local with closed point $x_0$. McAdam \cite{McAdam} showed that if $X$ is formally equidimensional and  $x_0 \in \mathrm{Ass}_{X}(R/\overline{\mathcal{I}^n})$ for some $n$, then the fiber of $\mathrm{Proj}(\mathcal{R}(\mathcal{I}))$ over $x_0$ is of codimension at most $1$. This result was generalized by Katz and Rice \cite{Katz2} using valuation theory to the case where $\mathcal{M}$ is contained in a free module $\mathcal{F}$ such that $\mathcal{M}$ has the same rank as $\mathcal{F}$ at the generic point of each irreducible component of $X$. As an immediate consequence of our main result we obtain the following far reaching generalization of these results: If $x_0 \in \mathrm{Ass}_{X}(\mathcal{N}^n/\overline{\mathcal{M}^n})$ for some $n$, then the irreducible components of the fiber of $\mathrm{Proj}(\mathcal{R}(\mathcal{M}))$ over $x_0$ are of codimension at most $1$.


In Thm.\ \ref{converse conormal} we prove a general converse to our main result. Suppose $X$ is local with closed point $x_0$. Say that a module $\mathcal{N}$ has {\it deficient analytic spread} if each irreducible component of the fiber of $\mathrm{Proj}(\mathcal{R}(\mathcal{N}))$ over $x_0$ is of codimension at least $2$. For example, if $\mathcal{N}$ is free and $\dim X \geq 2$, then $\mathcal{N}$ has deficient analytic spread. As a corollary of Thm.\ \ref{converse conormal} we derive the following converse of Thm.\ \ref{conormal}: assume $\mathcal{N}$ has deficient analytic spread and  $x_0 \not \in \mathrm{Ass}_{X}(\mathcal{N}^n/\overline{\mathcal{M}^n})$ for all large enough $n$, then $\mathcal{M}$ has deficient analytic spread. As another corollary to Thm.\ \ref{converse conormal} we prove in Thm.\ \ref{compl. desc.} a result which yields a  classification of the points of $\mathrm{Ass}_{X}(\mathcal{N}^n/\overline{\mathcal{M}^n})$ for arbitrary affine Noetherian scheme $X$. This result generalizes and strengthens results due to Burch \cite{Burch} for ideals, and Rees \cite{Rees}, Katz and Rice (see Thm.\ 3.5.1 in \cite{Katz2}) who treated the case where $\mathcal{N}$ is free and the rank of $\mathcal{M}$ at the generic point of each irreducible component of $X$ is equal to the rank of $\mathcal{N}$.  

The Rees algebra $\mathcal{R}(\mathcal{M})$ was essentially introduced by Zariski in App.\ IV to the second
volume of his book with Samuel \cite{ZS}. It was studied in greater generality by Rees in \cite{Rees}. The Rees algebra of a module is a generalization of the blowup algebra of $X$ with center an ideal $\mathcal{I} \subset R$. Considerable work on Rees algebras (see \cite {Rees2}, \cite{RK} and \cite{KT-Al}) was stimulated by the work of Gaffney (\cite{Gaf1} and \cite{Gaf2}) who needed a generalization of the blowup algebra for applications to equisingularity theory. In the geometric context $\mathcal{M}$ usually comes equipped with an embedding in a free module $\mathcal{F}$. Then $\mathcal{R}(\mathcal{M})$ is defined as the subalgebra of $\mathrm{Sym}(\mathcal{F})$ generated by the elements of $\mathcal{M}$. Although this definition is satisfactory in many situations, it has the disadvantage that it is neither functorial nor intrinsic. This issue was resolved in \cite{Eisenbud}. Furthermore, Eisenbud, Huneke and Ulrich show that their definition can be used to address questions related to integral dependence of modules. 

As another application of our main result, we recover, strengthen, and prove a sort of converse to an important result due to Kleiman and Thorup about integral dependence of modules. Assuming that $X$ is equidimensional and universally catenary, they show (see \cite{KT-Al} for the original proof and \cite{Thorup} for a much shorter and more elementary proof) that the  inverse image in $\mathrm{Proj}(\mathcal{R}(\mathcal{M}))$ of the locus in $X$ where $\mathcal{N}$ is not integral over $\mathcal{M}$ is of codimension one provided that the two modules are equal generically. Their result has found numerous applications to equisingularity theory (c.f. \cite{GK}, \cite{Gaf3}, and \cite{GR}). It is used as a crucial step in establishing that if a given equisingularity condition holds generically, then it holds everywhere. 

The original proof of Kleiman and Thorup involves an analysis of an intermediate object: the exceptional divisor of the blowup of $\mathrm{Proj}(\mathcal{R}(\mathcal{N}))$ with center the ideal $\mathcal{M}\mathcal{R}(\mathcal{N})$. In our treatment, we work directly on $\mathrm{Proj}(\mathcal{R}(\mathcal{M}))$. In this sense, our approach is similar in spirit to that of Simis, Ulrich and Vasconcelos \cite{SUV}, who analyze the integrality of $\mathcal{R}(\mathcal{N})$ over $\mathcal{R}(\mathcal{M})$ by reducing it to a local problem at codimension one primes of $\mathcal{R}(\mathcal{M})$. However, in contrast to \cite{SUV}, \cite{KT-Al} and \cite{KT}, our approach allows us to work in a more general setting: we don't require $X$ to be equidimensional, and we allow the ranks of $\mathcal{M}$ and $\mathcal{N}$ to vary across the generic points of the irreducible components of $X$. Finally, we strengthen the Kleiman--Thorup theorem by showing that the inverse image in $\mathrm{Proj}(\mathcal{R}(\mathcal{M}))$ of each irreducible component of the locus in $X$ where $\mathcal{N}$ is not integral over $\mathcal{M}$ is of codimension one in an irreducible component of $\mathrm{Proj}(\mathcal{R}(\mathcal{M}))$.

Many of the results in this paper are proved in greater generality: instead of working with Rees algebras of modules, we consider two finitely generated graded $R$-algebras $\mathcal{A} \subset \mathcal{B}$. In Rmk.\ \ref{general main} that follows the proof of our main result, we outline how one can derive Thm.\ \ref{conormal} in that setting at the expense of losing the explicit description of the genericity conditions on $h$. 


Here is a more technical overview of the contents. Assume $X$ is local with closed point $x_0$. Set $$\mathcal{R}^{\dagger}(\mathcal{M}) := R \oplus \overline{\mathcal{M}} \oplus \overline{\mathcal{M}^{2}} \oplus \cdots.$$ The problem of determining whether $x_0$ appears in the set $\bigcup_{n=1}^{\infty}\mathrm{Ass}_{X}(\mathcal{N}^{n}/\overline{\mathcal{M}^{n}})$ is closely related to that of
analyzing the kernel of the homomorphism
$$\phi_h \colon \mathcal{R}^{\dagger}(\mathcal{M})/h\mathcal{R}^{\dagger}(\mathcal{M}) \longrightarrow \mathcal{R}(\mathcal{N})/h\mathcal{R}(\mathcal{N})$$
where $h$ is a general element from the maximal ideal $\mathfrak{m}_{x_0}$. Denote the subscheme of zeroes of $h$ in $X$  by $H$.

In Sct.\ \ref{Supp and Fitt}, we prove that, for a generic choice of $h$, the components of the support of $\Ker \phi_h$ are properly contained in the components of $H$ using Fitting ideals and prime avoidance.

In Sct.\ \ref{Bertini} we prove a Bertini-type theorem. Call a morphism {\it extreme} if it maps every irreducible component of the source to a closed point or to an irreducible component of the target. Let $f \colon P \rightarrow X$ be a proper extreme morphism between a scheme $P$ and universally catenary local scheme $X$ with closed point $x_0$. We prove that $f$ is extreme on $f^{-1}H$ for a generic choice of $h \in \mathfrak{m}_{x_0}$. The genericity condition is that $H$ has to cut properly the irreducible components of the closed sets where the fiber dimension of $f$ jumps. That there are finitely many such sets follows from Chevalley's constructibility result. Recently, Gaffney and the author \cite{GRB} found applications of this Bertini-type theorem to equisingularity problems. 

In Sct.\ \ref{main results} we prove Thm.\ \ref{conormal}, the main result of our paper. Suppose $X$ is a local, universally catenary Noetherian scheme with closed point $x_0$ and that the dimension of each component of the fiber of $\mathrm{Proj}(\mathcal{R}(\mathcal{M}))$ over $x_0$ is of codimension at least $2$. Then $\Ker \phi_h$ vanishes for a generic choice of $h$ as specified in sections \ref{Supp and Fitt} and \ref{Bertini}. 

To prove Thm.\ \ref{conormal}, we need two technical lemmas. First Lem.\ \ref{key lemma} allows us to replace $R$, $\mathcal{M}^n$ and $\mathcal{N}^n$ with their respective completions along the maximal ideal $\mathfrak{m}_{x_0}$, because the associated points of $\mathcal{N}^{n}/\overline{\mathcal{M}^{n}}$ and the geometric hypothesis on $\mathrm{Proj}(\mathcal{R}(\mathcal{M}))$ behave well with passage to the completion. Because $R$ is complete local ring,  $R$ is Nagata (also known as universally Japanese or pseudo-geometric). But $\mathcal{R}(\mathcal{M})$ is of finite type over $R$. So $\mathcal{R}(\mathcal{M})$ is Nagata too. Then using Lem.\ \ref{reduction}, we pass to reduced structures. Because $\mathcal{R}(\mathcal{N})$ is reduced, $\mathcal{R}^{\dagger}(\mathcal{M})$ is then finite over $\mathcal{R}(\mathcal{M})$.

Next, we apply  Bertini's theorem for extreme morphisms to
$P:=\mathrm{Proj}(\mathcal{R}^{\dagger}(\mathcal{M}))$ and the structure morphism $f \colon P \rightarrow X$. The genericity conditions on $H$ ensure that the irreducible components of $f^{-1}H$ surject either onto $x_0$ or onto the irreducible components of $H$. The codimension hypothesis on the irreducible components of the fiber of $\mathrm{Proj}(\mathcal{R}(\mathcal{M}))$ over $x_0$ transfer to codimension hypothesis on the irreducible components of the fiber of $P$ over $x_0$ because $P$ has finite fibers over $\mathrm{Proj}(\mathcal{R}(\mathcal{M}))$. But the codimension hypothesis on irreducible components of the fiber of $P$ over $x_0$ guarantee that $f^{-1}H$ does not have irreducible components mapping to $x_0$. In other words, each irreducible component of $f^{-1}H$ surjects onto an irreducible component of $H$. As shown in Sct.\ \ref{Supp and Fitt}, our choice of $h$ guarantees that the components of the support of $\Ker \phi_h$, viewed as an $\mathcal{O}_{X,x_0}$-module, are properly contained in the irreducible components of $H$. But $\Ker \phi_h$ is also an ideal in $\mathcal{R}^{\dagger}(\mathcal{M})/h\mathcal{R}^{\dagger}(\mathcal{M})$. Thus each associated point of $\Ker \phi_h$ is the generic point of an embedded component of $f^{-1}H$. However, $\mathcal{R}^{\dagger}(\mathcal{M})$ is integrally closed in $\mathcal{R}(\mathcal{N})$. Using the Determinatal Trick Lemma, we show that  $f^{-1}H$ does not have embedded components. Hence $\Ker \phi_h$ vanishes. In particular, $x_0 \not \in \mathrm{Ass}_{X}(\mathcal{N}^{n}/\overline{\mathcal{M}^{n}})$ for all $n$.

We apply Thm.\ \ref{conormal} in a number of situations. First, in Sct.\ \ref{GKT} we strengthen and generalize the result of Kleiman and Thorup about the integral dependence of the modules $\mathcal{M} \subset \mathcal{N}$. Second, in Sct.\ \ref{finiteness and asymptotic stability} we apply Thm.\ \ref{conormal} and Chevalley's constructibility result  to obtain in Thm.\ \ref{ass points-int.cl.} an explicit geometric classification of the points in $\bigcup_{n=1}^{\infty}\mathrm{Ass}_{X}(\mathcal{N}^{n}/\overline{\mathcal{M}^{n}})$ when $X$ is universally catenary.

We prove a number of general results about finitely generated $R$-algebras that allow us to show that  $\mathrm{Ass}_{X}(\mathcal{N}^{n}/\mathcal{M}^{n})$ and $\mathrm{Ass}_{X}(\mathcal{N}^{n}/\overline{\mathcal{M}^{n}})$ are asymptotically stable. Notably, we derive these results from their more general counterparts about standard graded algebras. Prp.\ \ref{finite ass pnts} and Prp.\ \ref{asymptotics algebras}, which are due to Katz and Puthenpurakal \cite{Katz3} and Hayasaka \cite{Hayasaka}, respectively, show that if $\mathcal{A} \subset \mathcal{B}$ are standard graded $R$-algebras, then $\bigcup_{n=1}^{\infty}\mathrm{Ass}_{X}(\mathcal{B}_n/\mathcal{A}_n)$ is finite and the sequence $\mathrm{Ass}_{X}(\mathcal{B}_n/\mathcal{A}_n)$ is asymptotically stable. We give our own self-contained proofs of these results. Under mild assumptions on $\mathcal{B}$ we show in Thm.\ \ref{int. asym.} that 
$\mathrm{Ass}_{X}(\mathcal{B}_n/\overline{\mathcal{A}_n}) \subseteq \mathrm{Ass}_{X}(\mathcal{B}_{n+1}/\overline{\mathcal{A}_{n+1}})$ for each $n$, which generalizes results of Ratliff (see Thm.\ 2.4 in \cite{Ratliff84}) in the ideal case and Katz and Naude (see Thm.\ 2.1 in \cite{Katz1}) in the case when $\mathcal{B}$ is a standard graded polynomial ring. We also show that  $\bigcup_{i=1}^{\infty}\mathrm{Ass}_{X}(\mathcal{B}_n/\overline{\mathcal{A}_n})$ is finite by reducing to the case of an ideal following Katz and Naude \cite{Katz1}.

In Sct.\ \ref{converse results}, we prove Thm.\ \ref{converse conormal}, which is a fairly general converse  to Thm.\ \ref{conormal}. We introduce the notion of deficient analytic spread for modules and algebras and compared it with the more classical notion of nonmaximal analytic spread in Rmk.\ \ref{def. vs. anal.}. 
We apply these results and the result about the asymptotic stability of $\mathrm{Ass}_{X}(\mathcal{N}^{n}/\overline{\mathcal{M}^{n}})$ in Cor.\ \ref{converse to KT} to derive a sort of converse to the Kleiman--Thorup result.

Finally, assume that $\mathcal{N}$ has deficient analytic spread at each $x \in X$ with $\dim X,x \geq 2$. Our Thm.\ \ref{compl. desc.}, coupled with Cor.\ \ref{McAdam}, provides a necessary and sufficient condition for a point $x$ from $X$ to be in $\mathrm{Ass}_{X}(\mathcal{N}^{n}/\overline{\mathcal{M}^{n}})$. This generalizes  results by Burch \cite{Burch} for ideals, and Rees \cite{Rees}, Katz and Rice  \cite{Katz2} for the special case where $\mathcal{N}:=\mathcal{F}$ is free and the rank of $\mathcal{M}$ at the generic point of each irreducible component of $X$ is equal to  $\mathrm{rk}(\mathcal{\mathcal{F}})$. We finish by showing that if $\mathcal{M}$ is contained in a free module $\mathcal{F}$ such that $\mathcal{M}$ has rank $e$ at the generic point of each irreducible component of $X$ with  $\mathrm{rk}(\mathcal{F})>e$, and $\mathcal{M}$ has deficient analytic spread, then the set $\mathrm{Ass}_{X}(\mathcal{F}^{n}/\overline{\mathcal{M}^{n}})$ consists entirely of the generic points of the irreducible components of $X$ for each $n$. This situation occurs, for example, if $\mathcal{M}$ is the Jacobian module of a class of cone singularities with deficient conormal varieties.

{\bf  Acknowledgements.} I would like to thank my advisors Terence Gaffney and Steven Kleiman for many helpful discussions. I would like to express my gratitude to Steven Kleiman for reading carefully earlier drafts of this paper and supplying me with numerous invaluable comments and suggestions for improving its content and exposition. I would like to thank the anonymous referee for providing me with plenty of suggestions for improving the exposition of the paper as well as bringing to my attention the work of \cite{Katz3}. Much of this work was done during my visit at the trimester program in algebraic geometry at IMPA in 2015. I would like to thank Carolina Araujo, Ana-Maria Castravet and Eduardo Esteves, for organizing the program, for their hospitality, and making it possible for me to attend. I was supported by NSF conference grant DMS-1502154 ``Second Latin American School of Algebraic Geometry and Applications'', NEU College of Science travel grant and a research summer fellowship awarded by the Department of Mathematics at NEU.

\section{Definitions and Conventions}

For convenience of the reader we provide in this section some definitions and conventions used throughout the paper. Most of the results in sections $4, 5$ and $6$ are stated in geometric language. The main motivation for doing so is that the proof and statement of one of our crucial results Thm.\ \ref{vertical} can be easily stated and understood when phrased in geometric terms. Below we provide a simple dictionary between some basic geometric concepts used in the paper and their algebraic counterparts. 

Suppose $X=\mathrm{Spec}(R)$ where $R$ is commutative Noetherian ring with identity. Let $x$ be a point in $X$. We say $x$ is an {\it associated point} of $X$ if the ideal $\mathfrak{m}_x$ of $x$ is an associated prime in $R$. We say $x$ is an {\it embedded point} if $\mathfrak{m}_x$ is not a minimal associated prime of $R$. We say $x$ is an {\it associated point} of an $R$-module $\mathcal{L}$ if $\mathfrak{m}_x$ is an associated prime of $\mathcal{L}$. We denote by $\mathrm{Ass}_{X}(\mathcal{L})$ the set of all associated points of $\mathcal{L}$.

We say $X$ is {\it universally catenary} if $R$ is universally catenary (see Dfn.\ B.3.1 in \cite{Huneke}). We say $X$ is {\it Nagata} if $R$ is a  Nagata ring, i.e.\ if for every prime ideal $P$ in $R$ the integral closure of $R/P$ in any finite field extension of $\kappa(P)$ is module-finite over $R/P$ (see Ex.\ 9.6 in \cite{Huneke}).

We often work with a proper morphism $f\colon P \rightarrow X$ where $P$ and $X$ are schemes with $X$ Noetherian. In our applications $X:=\mathrm{Spec}(R)$, where $R$ is Noetherian, and $P:=\mathrm{Proj}(\mathcal{B})$ where $\mathcal{B}$ is a standard graded $R$-algebra. Evaluation of $f$ at points of $P$ corresponds to contraction of primes of $\mathcal{B}$ to $R$. The irreducible components of $P$ correspond to the minimal primes of $\mathcal{B}$. In the first sections of our paper $\mathcal{B}$ will usually take the form of a Rees algebra of a finitely generated $R$-module or an integral closure of it. 

If $V$ is a closed subscheme of $P$ and $V_i$ is an irreducible component of $V$ contained in an irreducible component $P_j$ of $P$, then we set $\mathrm{codim}(V_i,P_j)$ to be the Krull dimension of $\mathcal{O}_{P_j, \eta_{V_i}}$ where $\eta_{V_i}$ is the generic point of $V_i$. Alternatively, $\mathrm{codim}(V_i,P_j)$ is the height of the ideal of $V_i$ in the domain $\mathcal{B}/\mathcal{I}_{P_j}\mathcal{B}$ where $\mathcal{I}_{P_j}$ is the ideal 
of $P_j$ in $\mathcal{B}$. We define $\codim V_i$ as the Krull dimension of $\mathcal{O}_{P,\eta_{V_i}}$, or equivalently as the height of the ideal of $V_i$ in $\mathcal{B}$.

We say $X$ is {\it local with closed point} $x_0$ if $R$ is local with maximal ideal $\mathfrak{m}_{x_0}$. From now on till the end of the section suppose this is the case. 

Let $h$ be an element from $\mathfrak{m}_{x_0}$. We say that $H:=\mathbb{V}(H)$ {\it cuts properly} the irreducible components of a closed subset $X'$ of $X$ if $h$ avoids the minimal primes of the ideal of $X'$ in $R$. 

Let $V$ be a closed subset of $P$. Call $f$ {\it extreme} on $V$ if $f$ maps each irreducible component of $V$ either onto $x_0$ or onto an irreducible component of $f(V)$. 

We say that $\mathcal{B}$ has {\it deficient analytic spread} if the codimension (height) of each minimal prime ideal of $\mathfrak{m}_{x_0}\mathcal{B}$ is at 
least $2$. In Rmk.\ \ref{def. vs. anal.} we compare this notion with the classical notion of non-maximal analytic spread.

\section{Support and Fitting ideals}\label{Supp and Fitt}
Let $X$ be a local Noetherian scheme with closed point $x_0$. Denote the irreducible components of $X$ by $X_1, \ldots, X_r$.  Denote by $\eta_i$ the generic point of $X_i$. Suppose $\mathcal{M} \subset \mathcal{N}$ are finitely generated $\mathcal{O}_{X,x_0}$-modules which are free at $\eta_i$ of ranks $e_i$ and $p_i$ for each $i=1, \ldots, r$. Let $S$ be the union of the nonfree loci of $\mathcal{M}$ and $\mathcal{N}$ and $Z$ be the nonfree locus of $\mathcal{N}/\mathcal{M}$. For $h \in \mathcal{O}_{X,x_0}$ denote by $H$ the subscheme of zeroes of $h$ in $X$.

Let $\mathcal{N}^{n}$ be the $n$th homogeneous component of $\mathcal{R}(\mathcal{N})$  and let $\mathcal{M}^{n}$ be the image in $\mathcal{R}(\mathcal{N})$ of the $n$th homogeneous component of $\mathcal{R}(\mathcal{M})$.

\begin{proposition}\label{support} Assume $x_0$ is not an associated point of $X$ and $\dim X_i \geq 2$ for some $i$. There exists $h \in\mathcal{O}_{X,x_0}$ such that $H$ does not contain any associated point of $X$ and the irreducible components of $H$ are not contained in those of $Z$ and $S$. Furthermore, for each such $h$ the components of
\begin{equation}\label{powers support}
\mathrm{Supp}_{X}((h\mathcal{N}^{n} \cap \mathcal{M}^{n})/h\mathcal{M}^{n}) 
\end{equation}
are properly contained in the irreducible components of $H$ for any $n$. 

Moreover, if $X$ is reduced, we can replace $\mathcal{M}^{n}$ in $\rm{(\ref{powers support})}$ by any submodule of $\mathcal{N}^{n}$ that agrees with  $\mathcal{M}^{n}$ off $S$.
Finally, if $X$ is normal, then $\rm{(\ref{powers support})}$ can be strengthen to
\begin{equation}\label{supp eq.}
\mathrm{Supp}_{X}((h\mathcal{N}^{n} \cap \mathcal{M}^{n})/h\mathcal{M}^{n}) \subset H \cap S.
\end{equation}
\end{proposition}
\begin{proof} Let $\eta$ be the generic point of an irreducible component of $X$. To simplify notation assume $\mathcal{M}_{\eta}$ and $\mathcal{N}_{\eta}$ are free of ranks $e$ and $p$ respectively. First, we claim that $Z$ does not contain $\eta$. To verify this it is enough to show that $\mathcal{N}/\mathcal{M}\otimes k(\eta)$ is a vector space of dimension $p-e$. Equivalently, we want to show that $\mathcal{M}\otimes k(\eta)$ injects into $\mathcal{N}\otimes k(\eta)$. Therefore, it is enough to show that $\mathfrak{m}_{\eta}\mathcal{N}_{\eta} \cap \mathcal{M}_{\eta}=\mathfrak{m}_{\eta}\mathcal{M}_{\eta}$ where $\mathfrak{m}_{\eta}$ is the maximal ideal
of $\mathcal{O}_{X,\eta}$.

Let $f_1, \ldots, f_p$ be the generators of $\mathcal{N}_{\eta}$. An element $f$ from $\mathfrak{m}_{\eta}\mathcal{N}_{\eta}$ can be written as

\begin{equation}\label{free eq.1}
f = h_1f_1+ \cdots + h_pf_p
\end{equation}
where $h_i$ are elements from $\mathfrak{m}_{\eta}$.

Suppose $f$ belongs to $\mathcal{M}_{\eta}$.  Then we can write

\begin{equation}\label{free eq.2}
f = x_1m_1+ \cdots + x_em_e
\end{equation}
where $x_i$ are elements from $\mathcal{O}_{X,\eta}$ and $m_i$ are the generators of $\mathcal{M}_{\eta}$.

Because $\mathcal{O}_{X,\eta}$ is Artinian there exists a positive integer $q$ such that $\mathfrak{m}_{\eta}^{q} \neq 0$ and
$\mathfrak{m}_{\eta}^{q+1}=0$. Therefore, we can pick $\mathfrak{z}$ from $\mathfrak{m}_{\eta}^{q}$ such that $\mathfrak{z} \neq 0$.

Multiply both sides of (\ref{free eq.1})  and (\ref{free eq.2}) by $\mathfrak{z}$. Because $\mathfrak{z}h_i$ is in $\mathfrak{m}_{\eta}^{q+1}$, then $\mathfrak{z}h_i=0$. Thus
$\mathfrak{z}f=0$. Then (\ref{free eq.2})  and the fact that $\mathcal{M}_{\eta}$ is free force $\mathfrak{z}x_i=0$. This means that none of the $x_i$ is a unit
in $\mathcal{O}_{X,\eta}$ otherwise $\mathfrak{z}=0$ contradicting our choice of $\mathfrak{z}$. Therefore, $x_i$ is in $\mathfrak{m}_{\eta}$ for each $i$. Finally, (\ref{free eq.2}) implies that $f$ belongs to $\mathfrak{m}_{\eta}\mathcal{M}_{\eta}$ as desired.

By prime avoidance we can select $h$ from the maximal ideal of $\mathcal{O}_{X,x_0}$ such that $H$ cuts properly the irreducible components of $X$ and those of $S$ and $Z$ of positive dimension, and $h$ avoids the associated primes of $\mathcal{O}_{X,x_0}$. This choice of $h$ along with the assumption $\dim X_i \ge 2$ for some $i$ guarantees that the irreducible components of $H$ are not contained in those of $Z$ and $S$.

Let $x$ be a point in $X$ with $x \not \in S$ and $x \not \in Z$. Denote by $e$ and $p$ the ranks of $\mathcal{M}_x$ and $\mathcal{N}_x$, respectively. Because $x \not \in S$ it follows by Prp.\ 1.3 in \cite{Eisenbud} that $\mathcal{R}(\mathcal{M}_x)$ and  $\mathcal{R}(\mathcal{N}_x)$ are polynomial rings. Because $x \not \in Z$, then $\mathcal{R}(\mathcal{M}_x) \rightarrow \mathcal{R}(\mathcal{N}_x)$ is an inclusion. Thus $\mathcal{M}^{n}_{x}$ and $\mathcal{N}^{n}_{x}$ are free for each $n$ and the latter module is a direct summand of the former module. Denote by $e(n)$ and $p(n)$ the ranks of $\mathcal{M}^{n}_{x}$ and $\mathcal{N }^{n}_{x}$, respectively. 

Note that $h\mathcal{N}^{n} \cap \mathcal{M}^{n}=h\mathcal{M}^{n}$ off $H$. Therefore, to prove that the components of (\ref{powers support}) are properly contained in the components of $H$ it is enough to show that for the generic point $x$ of an irreducible component of $H$ the module $(h\mathcal{N} \cap \mathcal{M})/h\mathcal{M}$ is not supported at $x$. Fix such an $x$. Because $x \not \in S$ and $x \not \in Z$ we have 

$$\mathcal{N}^{n}_{x}/\mathcal{M}^{n}_{x} \simeq \mathcal{O}_{X,x}^{p(n)-e(n)}.$$
If there exists $f \in \mathcal{N}^{n}_{x}$ such that $hf \in \mathcal{M}^{n}_{x}$, then owing to the isomorphism above and the assumption that $H$ does not contain associated points of $X$, we get $f \in \mathcal{M}^{n}_x$ which implies $hf \in h\mathcal{M}^{n}_x$.

Suppose $X$ is reduced. Observe that in our treatment above we work locally at a point $x \notin S$. Because $X$ is reduced, the map
\begin{equation}\label{local inclusion}
i_x \colon \mathcal{R}(\mathcal{M}_x) \rightarrow \mathcal{R}(\mathcal{N}_x)
\end{equation}
is an inclusion of polynomial rings. Indeed, $\mathcal{M}$ is a direct summand of $\mathcal{N}$ locally at each $\eta_i$, so $\mathrm{ker}(i_x) \subset \mathrm{nil}(\mathcal{R}(\mathcal{M}_x))$. But $X,x$ is reduced, so $\mathrm{nil}(\mathcal{R}(\mathcal{M}_x))=0$ (see Prp.\ \ref{structure morphism} $\rm{(1)}$). In particular, for each $n$ the modules $\mathcal{M}^{n}_{x}$ and $\mathcal{N}^{n}_{x}$ are $\mathcal{O}_{X,x}$-free of ranks $e(n)$ and $p(n)$. So, if $X$ is reduced it does not make any difference if we replace $\mathcal{M}^{n}$ with any other submodule of $\mathcal{N}^{n}$ locally equal to $\mathcal{M}^{n}$ off $S$.

Finally, suppose $X$ is normal. As before, select $H$ so that it cuts properly the irreducible components of $X$ and $Z$. Then $H$ does not contain an associated point of $X$, and $Z$ does not contain an associated point of $H$. Let $x$ be a point from $H$ with $x \not \in S$. Then $\mathcal{M}_x$ and $\mathcal{N}_x$ are free of ranks $e_i$ and $p_i$ for some $i$. So $x \in X_{i}'$ and $x \not \in X_{k}'$ for $k \neq i$. Because $i_x$ in (\ref{local inclusion}) is injective,  $\mathcal{M}^{n}_{x}$ and $\mathcal{N}^{n}_{x}$ are free of ranks $e_{i}(n)$ and $p_{i}(n)$ respectively. Suppose there exists $f \in \mathcal{N}^{n}_{x}$ such that $hf \in \mathcal{M}^{n}_{x}$. Then there exist $x_j \in \mathcal{O}_{X,x}$ such that
\begin{equation}\label{z.div eq}
\sum_{j=1}^{e_{i}(n)}x_{j}m_{j}=hf
\end{equation}
where $\langle m_1, \ldots, m_{e_{i}(n)} \rangle$ is a basis for $\mathcal{M}^{n}_{x}$. Write each $m_j$ in the basis for $\mathcal{N}^{n}_x$ so that (\ref{z.div eq}) becomes a system of $p_{i}(n)$ equations in $e_{i}(n)$ unknowns. By Cramer's rule for each $j$ and each maximal minor $m_{e_{i}(n) \times e_{i}(n)}$  of the coefficient matrix of (\ref{z.div eq}) we have
\begin{equation}\label{cramer eq}
m_{e_{i}(n) \times e_{i}(n)}x_j = hu_j
\end{equation}
where $u_j \in \mathcal{O}_{X,x}$ depend on $m_{e_{i}(n) \times e_{i}(n)}$ and $f$.


Next we claim that the radicals of $\mathrm{Fitt}_{p_i-e_i}(\mathcal{N}_x/\mathcal{M}_x)$ and $\mathrm{Fitt}_{p_{i}(n)-e_{i}(n)}(\mathcal{N}_{x}^{n}/\mathcal{M}_{x}^{n})$ are the same for any $n$. Indeed, two ideals have the same radicals if and only if their varieties consist of the same points. Because the formation of Fitting ideals commutes with base change we can pass to residue fields of points and check the statement over a field (see the display at the bottom of
p.\ 235 in \cite{Bruns}).

Since the formation of Fitting ideals commutes with base change, for each $n$ we have
\begin{equation}\label{Fitt base change}
\mathrm{Fitt}_{p_{i}(n)-e_{i}(n)}(\mathcal{N}^{n}/\mathcal{M}^{n})\mathcal{O}_{X,x} = \mathrm{Fitt}_{p_{i}(n)-e_{i}(n)}(\mathcal{N}^{n}_{x}/\mathcal{M}^{n}_{x}).
\end{equation}

Because $Z$ does not contain an associated point of $H$, then by (\ref{Fitt base change}) and Prp.\ \ref{Fitt} the ideal $\mathrm{Fitt}_{p_{i}-e_{i}}(\mathcal{N}_x/\mathcal{M}_x)$ and hence $\mathrm{Fitt}_{p_{i}(n)-e_{i}(n)}(\mathcal{N}_{x}^{n}/\mathcal{M}_{x}^{n})$ avoid the associated primes of $\mathcal{O}_{X,x}/h\mathcal{O}_{X,x}$. By prime avoidance for each $n$ there exists $w_n \in \mathrm{Fitt}_{p_{i}(n)-e_{i}(n)}(\mathcal{N}^{n}_{x}/\mathcal{M}^{n}_{x})$ such that $w_n$ avoids the associated primes of $\mathcal{O}_{X,x}/h\mathcal{O}_{X,x}$ and hence those of $\mathcal{O}_{X,x}$ because $h \not \in \mathrm{z.div}(\mathcal{O}_{X,x})$.

Write $w_n$ in terms of the minors $m_{e_{i}(n) \times e_{i}(n)}$ which are generators for $\mathrm{Fitt}_{p_{i}(n)-e_{i}(n)}(\mathcal{N}^{n}_{x}/\mathcal{M}^{n}_{x})$. Then by (\ref{cramer eq}) we can write
$$w_{n}x_j=hu_{j}'$$
for each $x_j$, where $u_{j}' \in \mathcal{O}_{X,x}$. However, $w_{n} \not \in \mathrm{z.div}(\mathcal{O}_{X,x}/h \mathcal{O}_{X,x}).$ Thus
$$x_j =hu_{j}''$$
for each $j$ with $u_{j}'' \in \mathcal{O}_{X,x}$. Substitute the last equations for $x_j$ back in (\ref{z.div eq}).
We get $hf \in h\mathcal{M}^{n}_{x}$. Thus $h\mathcal{N}^{n}\cap \mathcal{M}^{n}=h\mathcal{M}^{n}$ at $x$ which proves (\ref{supp eq.}).
\end{proof}
The rank hypothesis on $\mathcal{M}$ and $\mathcal{N}$ allows us to find easily equations for $S$ and $Z$. 

Suppose $\mathcal{K}$ is an $\mathcal{O}_{X,x_0}$-module of rank $k_i$ at $\eta_i$ for $i=1, \ldots, r$. For each $i$ define 
 $X_{i}'$ as the union of all irreducible
components $X_j$ such that $k_i = k_j$. Set $W_i: = X_{i}' \cap
\mathbb{V}(\mathrm{Fitt}_{k_i}(\mathcal{K}))$ and let $W$ be the union of the $W_{i}$s. 

\begin{proposition}\label{Fitt}
The nonfree locus of $\mathcal{K}$ is $W$. 
\end{proposition}
\begin{proof}
Suppose $x \not \in W$ and $x \in X_{i}'$ for some $i$. 
Because $x \not \in W_i$, then $\mathcal{K}_x$ can be generated by 
$k_i$ generators but not fewer
because $\mathcal{K}$ is free of rank $k_i$ at the generic points of the
irreducible components of $X_{i}'$. We claim that the locus in
$\mathrm{Spec}(\mathcal{O}_{X,x})$ where  $\mathcal{K}$ can be generated by fewer generators is
empty. This will imply at once that  $\mathcal{K}_x$ is free, i.e.\ $x$ is not in
the nonfree locus of $\mathcal{K}$.

Suppose there exists $y \in \mathrm{Spec}(\mathcal{O}_{X,x})$ such that $\mathcal{K}_y$ can be
generated by fewer generators. Then $y$ and therefore $x$ must belong to
$X_{j}'$ for some $j$ such that $k_j < k_i$. But $x$ is not in $W$, so
$x$ is not in $W_j$.  Hence $\mathcal{K}_x$ can be generated by fewer than $k_i$
generators which is impossible as observed above.

Conversely, suppose $x$ is not in the nonfree locus of $\mathcal{K}$ and $x \in
X_{i}'$ for some $i$. Then $\mathcal{K}_x$ is free of rank $k_i$ because  $\mathcal{K}$ 
is free of rank $k_i$ at the generic points of the irreducible
components of $X_{i}'$. Thus $\mathcal{K}_x$ can be generated by $k_i$
elements. So $x \not \in W_i$. We claim $x \not \in W_j$ for $j \neq
i$. Indeed, if $x \in X_{j}'$, then $\mathcal{K}$ will have rank $k_j$ at the
generic points of the irreducible components of $X_{j}'$ which is
impossible. Thus $x \not \in W$.
\end{proof}

To obtain equations for $S$ apply Prp.\ \ref{Fitt} with $\mathcal{K}=\mathcal{M}$ and $\mathcal{K}=\mathcal{N}$
to get the nonfree loci of $\mathcal{M}$ and $\mathcal{N}$. Then $S$ will be their union. Apply Prp.\ \ref{Fitt} with $\mathcal{K}=\mathcal{N}/\mathcal{M}$ and $k_i =p_i-e_i$ to obtain equations for $Z$. For a more general and global version of Prp.\ \ref{Fitt} see the material on pp.\ 55--57 in \cite{Mumford}.

\section{Bertini's theorem for extreme morphisms}\label{Bertini}
Assume $X$ is Noetherian universally catenary local scheme with closed point $x_0$. Denote by $X_1, \ldots, X_r$ the irreducible components of $X$. Let $P$ be a scheme and $f \colon P \rightarrow X$ be a proper morphism. For any closed set $V$ in $P$ denote by $V_{\mathrm{vert}}$ the union of the irreducible components of $V$ that are mapped to $x_0$, and by $V_{\mathrm{hor}}$ the union of those irreducible components that surject onto irreducible components of $f(V)$. Call the components of $V_{vert}$ {\it vertical} and those of $V_{\mathrm{hor}}$ {\it horizontal}. 
\begin{definition}
Call $f$ {\it extreme} on $V$ if $f$ maps each irreducible component of $V$ either onto $x_0$ or onto an irreducible component of $f(V)$.
\end{definition} 
Denote by  $\mathfrak{m}_{x_0}$ the ideal of $x_0$ in $\mathcal{O}_{X,x_0}$. For $h \in \mathfrak{m}_{x_0}$ denote by $H$ the subscheme of zeroes of $h$ in $X$.

For each $X_i$ denote by $P_i$ the union of components of $P$ that map to $X_i$. Denote by $f_i$ the restriction of $f$ to $P_i$. For each integer $k \ge 0$, set
\begin{equation}\label{S(i)}
S_i(k) := \{x\in X_i \ | \ \dim f_{i}^{-1}x \ge k \}.
\end{equation}
\begin{lemma}\label{S}
For each $i$ the sets $S_{i}(k)$ are closed and nonempty for finitely many $k$. Let $I$ be an ideal in $\mathcal{O}_{X,x_0}$ that avoids the minimal primes of the ideal of each $S_{i}(k)$ of positive dimension. Then there exists $h \in I$ such that $H$ cuts properly the components of all $S_{i}(k)$ of positive dimension.
\end{lemma}
\begin{proof}
Each $S_{i}(k)$ is closed in $X$ because $f$ and hence $f_i$ is proper (see the material in Chevalley's Thm.\ 13.1.3 and Cor.\ 13.1.5 in \cite{Grothendieck1}). Note that there are finitely many nonempty $S_{i}(k)$ because $P$ is of finite type over a Noetherian scheme and hence Noetherian. 
By prime avoidance we can choose $h$ from $I$ such that $h$ avoids the minimal primes of $S_{i}(k)$.
\end{proof}
The main result of this section is a Bertini's theorem for extreme morphisms. Denote by $P_{\mathrm{hor}}$ the union of irreducible components of $P$ that surject to irreducible components of $X$. Recall that for each irreducible component $X_i$ of $X$, we denote by $P_i$ the union of components of $P$ that map to $X_i$.
\begin{theorem}\label{vertical} Assume $P_{\mathrm{hor}} \cap P_i$ is equidimensional for each $i$ and assume $f \colon P \rightarrow X$ is a surjective and  extreme morphism on $P$. Let $h \in \mathfrak{m}_{x_0}$ such that $H$ cuts properly the irreducible components of positive dimension of $S_{i}(k)$ and $X_i \cap X_j$. Then $f$ is surjective and extreme on $f^{-1}H$.
\end{theorem}
\begin{proof} Since $f$ surjects $P$ onto $X$ by assumption and $f^{-1}H$ is a closed subscheme of $P$, it follows that $f$ surjects $f^{-1}H$ on $H$.

Let $W_1$ be a component of $f^{-1}H$. If $W_1$ is contained in $P_{\mathrm{vert}}$, then $W_1$ is vertical and we are done. So we can assume that $W_1 \subset P_{\mathrm{hor}}$ and replace $P$ by $P_{\mathrm{hor}}$.

For an irreducible component $X_i$ of $X$ set $a_i:= \dim P_i$ and $b_i:=\dim X_i$. Set $c_i:=a_i-b_i$. By assumption $f_i$ is surjective and proper. Then by the dimension formula (see \cite[\href{http://stacks.math.columbia.edu/tag/02JX}{Tag 02JX}]{Stacks} or Lem.\ 3.1 (ii) in \cite{KT-Al}) the dimension of the fiber of $f_i$ over the generic point of $X_i$ has dimension $c_i$. Consequently, $$\dim f_{i}^{-1}S_{i}(k) \leq a_i-1 \ \text{for} \ k \geq c_i+1$$ because the irreducible components of $f_{i}^{-1}S_{i}(k)$ are properly contained in those of $P_i$ for $k \geq c_{i}+1$. Next, the dimension formula applied for a component of $f_{i}^{-1}S_{i}(k)$ that maps onto a component of $S_{i}(k)$ yields
\begin{equation}\label{key inequality}
\dim S_{i}(k) \leq a_{i}-k-1 \ \text{for} \ k \geq c_{i}+1.
\end{equation}

Let $p_1$ be a closed point in an irreducible component $Q_1$ of $P_1$. Then $f_{1}(p_1)=x_0$ because $f_{1}$ is closed and $x_0$ is the only closed point in $X_1$. We want to show that $\dim Q_1= \dim \mathcal{O}_{Q_1,p_1}$. 

Denote by $k(p_1)$ and $k(x_0)$ the residue fields of $p_1$ and $x_0$, and denote by $R(Q_1)$ and $R(X_1)$ the function fields of $Q_1$ and $P_1$. Because $X$ is universally catenary, so is $X_1$. The dimension formula (\cite[\href{http://stacks.math.columbia.edu/tag/02JU}{Tag 02JU}]{Stacks} or Thm.\ B.5.1 in \cite{Huneke}) yields

\begin{equation}\label{tr.deg}
\dim \mathcal{O}_{Q_1,p_1} + \mathrm{tr.deg}_{k(x_0)} k(p_1) = \dim \mathcal{O}_{X_1,x_0} + \mathrm{tr.deg}_{R(X_1)}R(Q_1).
\end{equation}

By the Weak Nullstellensatz, $k(p_1)$ is algebraic over $k(x_0)$. Hence $\mathrm{tr.deg}_{k(x_0)} k(p_1)=0$. Also, because $f_1$ is closed, then $\dim \mathcal{O}_{X_1,x_0} + \mathrm{tr.deg}_{R(X_1)}R(Q_1) = \dim P_1$ by  \cite[\href{http://stacks.math.columbia.edu/tag/02JX}{Tag 02JX}]{Stacks}. Finally, we get $\dim \mathcal{O}_{Q_1,p_1}=\dim P_1$ as desired.

By Lem.\ \ref{S} we can choose $h \in \mathfrak{m}_{x_{0}}$ that avoids the minimal primes $\mathcal{O}_{S_{i}(k),x_0}$ for each $S_{i}(k)$ of positive dimension. In particular, this choice of $h$ guarantees that each irreducible component of $H \cap X_i$ is an irreducible component of $H$, because $S_{i}(c_i)=X_i$ and $H$ cuts properly the irreducible components of $X_i \cap X_j$ of positive dimension. Let $W_1$ be an irreducible component of $f^{-1}H$. Without lost of generality assume $W_1$ lies in an irreducible component $Q_1$ of $P_1$. 

Let $p_1$ be a closed point of $W_1$ with $\dim W_1 = \dim \mathcal{O}_{W_1,p_1}$. Consider the ring map $f_{1}^{\sharp}\colon \mathcal{O}_{X_1,x_0} \longrightarrow \mathcal{O}_{Q_1,p_1}.$ Since $h$ avoids the minimal primes of $\mathcal{O}_{X,x_{0}}$, then its image in $\mathcal{O}_{X_1,x_0}$ is nonzero. Also, $f_{1}^{\sharp}(h)$ is nonzero because by hypothesis $f_1$ maps $P_1$ onto $X_1$. Denote by $\mathcal{I}_{W_1}$ the ideal of $W_1$ in $\mathcal{O}_{Q_1,p_1}$. By assumption $\mathcal{I}_{W_1}$  is a minimal prime of $f_{1}^{\sharp}(h)$. Then by Krull's principal ideal theorem $\mathrm{ht}(\mathcal{I}_{W_1})=1$.  Since $f$ is of finite type, then so is $f_1$. Because $X_1$ is universally catenary, then $\mathcal{O}_{Q_1,p_1}$ is catenary. But $\dim \mathcal{O}_{Q_1,p_1}=a_1$ as shown above. Thus $\mathrm{ht}(\mathcal{I}_{W_1})+\dim W_1 = a_1$. Hence $\dim W_1 = a_{1}-1$. 


Assume that the generic fiber of $f$ over $f(W_1)$ is of dimension $k$. Then the dimension formula implies that $\dim f(W_1) = a_1-k-1$.

Suppose $k=c_1$. Then $\dim f(W_1)=b_{1}-1$. As $f(W_1)$ is irreducible closed set contained in the closed set $H \cap X_1$ of dimension $b_1-1$, then $f(W_1)$ must be a component of $H \cap X_1$ and hence of $H$. So, $W_1$ is a horizontal component of $f^{-1}(H)$.

Suppose $k \ge c_{1}+1$. By assumption $f(W_1) \subset S_{1}(k)$. Assume $\dim S_{1}(k)>0$. By construction $H$ intersects properly the irreducible components of $S_{1}(k)$. Therefore, $f(W_1)$ is strictly contained in an irreducible component of $S_{1}(k)$ and hence $\dim f(W_1) < \dim S_{1}(k)$, i.e. we must have $$\dim S_{1}(k) > a_{1}-k-1$$ reaching a contradiction with (\ref{key inequality}). Therefore, $\dim S_{1}(k)=0$, so $S_{1}(k)=x_0$, and $W_1$ is a vertical component of $f^{-1}H$.
\end{proof}
\begin{remark}\label{dense} \rm Assume $X$ is essentially of finite type over a field  $\Bbbk$. In Lem.\ \rm{\ref{S}} fix a minimal generating set $(i_1, \ldots, i_n)$ for $I$. Because $\Bbbk$ is a field we can associate in a unique way each $\Bbbk$-linear combination of the generators of $I$ with a point in $\mathbb{A}_{\Bbbk}^{n}$. If in addition $\Bbbk$ is infinite, then there exists a Zariski open dense subset $U_h$ in $\mathbb{A}_{\Bbbk}^{n}$ of elements $h$ such that $H$ cuts properly the components of each $S(i)$ of positive dimension.
\end{remark}

\section{The main result}\label{main results}
Let $X:=\mathrm{Spec}(R)$ be an affine Noetherian scheme. Denote by $X_1, \ldots, X_r$ its irreducible components and for each $i$ denote by $\eta_i$ the generic point of $X_i$. Let $\mathcal{M} \subset \mathcal{N}$ be a pair of finitely generated $R$-modules. Assume 
either 
\begin{enumerate}
\item[(1)] $\mathcal{M}_{\eta_i}$ and $\mathcal{N}_{\eta_i}$ are free $\mathcal{O}_{X,\eta_i}$-modules of ranks $e_i$ and $p_i$,
or
\item[(2)] $\mathcal{M}$ and $\mathcal{N}$ are contained in a free $R$-module $\mathcal{F}$.
\end{enumerate} At the end of this section we explain the necessity of these hypothesis.

Assume $\mathcal{N}$ is contained in a free module $R$-module $\mathcal{F}$. Although the canonical map $\mathcal{N}^{**} \to \mathcal{F}$ needn't be injective,
nevertheless the inclusion $\mathcal{N} \hookrightarrow \mathcal{F}$ factors through the canonical
surjection from $\mathcal{N}$ to its image in  $\mathcal{N}^{**}$, termed the {\it torsionless quotient}
of $\mathcal{N}$ in \cite{Eisenbud}, and so this surjection is an isomorphism.  Denote by $\mathcal{R}_{\mathcal{F}}(\mathcal{N})$ the subalgebra of $\mathrm{Sym}(\mathcal{F})$ generated by the elements of $\mathcal{N}$. Then by Prp.\ 1.8 \cite{Eisenbud} $\mathcal{R}(\mathcal{N})_{\mathrm{red}}$ and $\mathcal{R}_{\mathcal{F}}(\mathcal{N})_\mathrm{red}$ are the same. The same applies for $\mathcal{R}(\mathcal{M})_{\mathrm{red}}$ and $\mathcal{R}_{\mathcal{F}}(\mathcal{M})_\mathrm{red}$. Denote by $e_i$ and $p_i$ the ranks of the images of $\mathcal{M}$ and $\mathcal{N}$ in $\mathcal{F}\otimes_{\mathcal{O}_{X}} \mathcal{O}_{{X}_{\mathrm{red}}}$ at the generic point of $(X_i)_\mathrm{red}$, respectively.


Set $B := \mathrm{Proj}(\mathcal{R}(\mathcal{M}))$. Denote by $b$ the structure morphism from $B$ to $X$. Let $B_{i}$ be the union of irreducible components of $B$ that lie over $X_i$ and denote by  $b_i$ the restriction of $b$ to $B_i$. Finally, let $x_0$ be a point in $X$. Denote by $D_1, \ldots, D_q$ the irreducible components of $b^{-1}x_0$.

Suppose $D_j$ is contained in some $B_i$. Denote by $\codim (D_j,B_i)$ the Krull dimension of the local ring $\mathcal{O}_{B_i,\eta}$ where $\eta$ is the generic point of $D_j$. Denote by $\codim D_j$ the Krull dimension of $\mathcal{O}_{B,\eta}$.
\begin{proposition}\label{structure morphism} The following hold:
\begin{enumerate}
\item [{\rm(1)}] Each $B_i$ is irreducible and maps surjectively onto $X_i$.
\item [{\rm(2)}] The morphism $b$ is proper and $\dim B_i = \dim X_i + e_{i}-1$.
\item [{\rm(3)}] If $X$ is local scheme of positive dimension with closed point $x_0$, then
    $$\codim (D_j,B_i) \geq 1$$ for each $B_i$ that contains $D_j$. Additionally, if $X$ is universally catenary, then $$\codim (D_j,B_i)=\dim B_i- \dim D_j.$$
\item [{\rm(4)}] The Rees algebra $\mathcal{R}(\mathcal{M})$ and its image in $\mathcal{R}(\mathcal{N})$ have the same reduced structures. Each of the statements {\rm(1)}, {\rm(2)} and {\rm(3)} remains valid after replacing  $\mathcal{R}(\mathcal{M})$ with its image in $\mathcal{R}(\mathcal{N})$. Furthermore, after the replacement $\codim D_j$ remains the same.
\end{enumerate}
\end{proposition}
\begin{proof}
By Prp.\ 1.5 of \cite{Eisenbud}, there is an inclusion-preserving bijection between the associated primes of 
$\mathcal{R}(\mathcal{M})$ and $R$ given by contraction. So each minimal prime of $\mathcal{R}(\mathcal{M})$ contracts to a minimal prime of $R$. Thus ${\rm(1)}$ holds.

Consider $\rm{(2)}$. Note $b$ is projective, so proper. Pass to $X_{\mathrm{red}}$ and $\mathcal{R}(\mathcal{M})_{\mathrm{red}}$ if necessary. By Prp.\ 1.3 of \cite{Eisenbud} the formation of 
the Rees algebra of a finitely generated module commutes with flat base change. So $\mathcal{R}(\mathcal{M})_{\eta_i} =\mathcal{R}(\mathcal{M}_{\eta_i})$. By assumption $\mathcal{M}_{\eta_{i}}$ is $\mathcal{O}_{X_i,\eta_i}$-free of rank $e_i$. Thus the dimension of the fiber of $b_i$ over $\eta_{i}$ is $e_{i}-1$. By $\rm{(1)}$  $B_i$ is irreducible. Then the dimension formula (see Lem.\ 3.1 (ii) in \cite{KT-Al}) yields $\dim B_i = \dim X_i+e_i-1$  which proves \rm{(2)}.

For $\rm{(3)}$ notice that by \rm{(1)} each component $D_j$ of  $b^{-1}(x_0)$ is contained properly in an irreducible component of $B$  because $x_0$ is not an irreducible component of $X$. Therefore, $\codim (D_j,B_i) \geq 1.$ The second part of the statement follows from an analogous computation to the one we did in (\ref{tr.deg}) this time
around applied for $B_i$ and the generic point  of $D_j$. 

Consider $\rm{(4)}$. Suppose $\mathcal{M}_{\eta_i}$ and $\mathcal{N}_{\eta_i}$ are $\mathcal{O}_{X,\eta_i}$-free. In the proof of Prp. \ref{support} we showed that $\mathcal{R}(\mathcal{M})$ and its image in $\mathcal{R}(\mathcal{N})$ are isomorphic at each $\eta_i$. In $\rm{(1)}$ we showed that the generic point of each $B_i$ maps to $\eta_i$. Hence the kernel of the homomorphism $\mathcal{R}(\mathcal{M}) \rightarrow \mathcal{R}(\mathcal{N})$ is nilpotent. 

Suppose $\mathcal{M}$ and $\mathcal{N}$ are contained in a free $R$-module $\mathcal{F}$. Then $\mathcal{R}(\mathcal{M}) \rightarrow \mathcal{R}(\mathcal{N})$ induces a homomorphism $\mathcal{R}(\mathcal{M})_{\mathrm{red}} \rightarrow \mathcal{R}(\mathcal{N})_{\mathrm{red}}$. But $\mathcal{R}(\mathcal{M})_{\mathrm{red}} \simeq \mathcal{R}_{\mathcal{F}}(\mathcal{M})_{\mathrm{red}}$ and $\mathcal{R}(\mathcal{N})_{\mathrm{red}} \simeq \mathcal{R}_{\mathcal{F}}(\mathcal{N})_{\mathrm{red}}$ as discussed in the beginning of this section. Also, $\mathcal{R}_{\mathcal{F}}(\mathcal{M})_{\mathrm{red}} \hookrightarrow \mathcal{R}_{\mathcal{F}}(\mathcal{N})_{\mathrm{red}}$ because the two algebras are generated by the images of $\mathcal{M}$ and $\mathcal{N}$ in the polynomial ring $\mathrm{Sym}(\mathcal{F})_{\mathrm{red}}$. Thus once again we obtain
that the kernel of $\mathcal{R}(\mathcal{M}) \rightarrow \mathcal{R}(\mathcal{N})$ is nilpotent. 

Therefore, in all cases $\mathcal{R}(\mathcal{M})$ and its image in $\mathcal{R}(\mathcal{N})$ have the same reduced structures. Thus, \rm{(1)}, \rm{(2)} and \rm{(3)} remain valid after the replacement. Therefore, the fiber over $x_0$ of the $\mathrm{Proj}$ of the image of $\mathcal{R}(\mathcal{M})$ in $\mathcal{R}(\mathcal{N})$ has the same reduced structure as $b^{-1}x_0$; hence, the codimension of $D_j$ remains the same.
\end{proof}

Next, for each positive integer $n$ let $\mathcal{M}^{n}$ be the image of the $n$th graded component of $\mathcal{R}(\mathcal{M})$ in $\mathcal{R}(\mathcal{N})$. Define the integral closure $\overline{\mathcal{M}^{n}}$ of $\mathcal{M}^{n}$ in $\mathcal{N}^{n}$ to be the submodule of $\mathcal{N}^{n}$ that consists of all elements $g$ satisfying an equation of integral dependence
$$g^{u}+m_1g^{u-1}+ \cdots + m_u = 0$$
in $\mathcal{R}(\mathcal{N})$ where $u$ is a positive integer and $m_i$ belongs to $\mathcal{M}^{ni}$ for each $i$.

The integral closure of $\mathcal{R}(\mathcal{M})$ in $\mathcal{R}(\mathcal{N})$ is the subalgebra of $\mathcal{R}(\mathcal{N})$ generated by all elements that satisfy an equation of integral dependence over $\mathcal{R}(\mathcal{M})$ . In general, the integral closure of $\mathcal{R}(\mathcal{M})$ may not be a Noetherian ring. To ensure that this is not the case, in our main result we will reduce to the case when $X$ is Nagata scheme. Furthermore, we will need to show that we can replace $X$ and all the algebras involved with their reductions. The following lemma enables us to do so. It describes the connection between associated points, integral closure and reductions.

Assume $\mathcal{R}$ is a commutative ring. Let $\mathcal{A} \subset \mathcal{B}$ be $\mathbb{N}_0$-graded. Denote by $\mathcal{A}_n$ and $\mathcal{B}_n$ their $n$th graded pieces. Assume $\mathcal{A}_0=\mathcal{B}_0=\mathcal{R}$. Let $\overline{\mathcal{A}_n}$ be the integral closure of $\mathcal{A}_n$ in $\mathcal{B}_n$. Denote by $(\mathcal{A}_{n})_{\text{red}}$, $(\overline{\mathcal{A}_n})_{\text{red}}$ and $(\mathcal{B}_{n})_{\text{red}}$ the images of $\mathcal{A}_n$, $\overline{\mathcal{A}_n}$ and $\mathcal{B}_n$ in $\mathcal{B}_{\mathrm{red}}$, respectively.

\begin{lemma}\label{reduction}
The following hold:
\begin{enumerate}
\item $(\overline{\mathcal{A}_n})_{\mathrm{red}}=\overline{(\mathcal{A}_n)_{\mathrm{red}}}.$
\item Suppose $h \in \mathcal{R}$. Then
$$h \not \in \mathrm{z.div}(\mathcal{B}_{n}/\overline{\mathcal{A}_n})  \ \text{if and only if} \ h \not \in \mathrm{z.div}((\mathcal{B}_{n})_{\mathrm{red}}/(\overline{\mathcal{A}_{n}})_{\mathrm{red}}).$$
\item Suppose $\mathcal{R}$ is local Noetherian ring with maximal ideal $\mathfrak{m}$. Then 
$$\mathfrak{m} \not \in \mathrm{Ass}_{\mathcal{R}}(\mathcal{B}_{n}/\overline{\mathcal{A}_n}) \ \text{if and only if} \ \mathfrak{m}_{\mathrm{red}} \not \in \mathrm{Ass}_{\mathcal{R}_{\mathrm{red}}}((\mathcal{B}_{n})_{\mathrm{red}}/(\overline{\mathcal{A}_{n}})_{\mathrm{red}}).$$
\end{enumerate}
\end{lemma}
\begin{proof} Consider $\rm(1)$. By persistence of integral closure $(\overline{\mathcal{A}_n})_{\text{red}} \subset \overline{(\mathcal{A}_n)_{\text{red}}}$. Let $b_n$ be an element from $\mathcal{B}_n$ whose image in $(\mathcal{B}_n)_{\text{red}}$ is in $\overline{(\mathcal{A}_n)_{\text{red}}}$. Then
$$b_{n}^{l}+a_{1}b_{n}^{l-1}+\cdots +a_{l}=b$$
where $l$ is a positive integer, each $a_{i}$ is from $\mathcal{A}$ and $b$ is a nilpotent element from $\mathcal{B}$, i.e. $b^{u}=0$ for some positive integer $u$. Raising the last equation to the power $u$ gives an equation of integral dependence for $b_n$. Hence $b_n \in \overline{\mathcal{A}_n}$. Thus $\overline{(\mathcal{A}_n)_{\text{red}}} \subset (\overline{\mathcal{A}_n})_{\text{red}}$ which proves the equality in $\rm(1)$.

Consider $\rm(2)$. Suppose $hb_n \in \overline{\mathcal{A}_n}$ for some $b_n \in \mathcal{B}_n$ with $b_n \not \in \overline{\mathcal{A}_n}$. Because $b_n \not \in \overline{\mathcal{A}_n}$, then $b_n$ is not nilpotent in $\mathcal{B}$. Hence its image $\widetilde{b_n}$ in $(\mathcal{B}_n)_{\text{red}}$ is nonzero. Also, $\widetilde{b_n}$ is not in $(\overline{\mathcal{A}_n})_{\text{red}}$ because otherwise
$b_n = a+b$ where $a \in \overline{\mathcal{A}_n}$ and $b$ is nilpotent in $\mathcal{B}$, and hence $b_n$ would be integral over $\mathcal{A}_n$ which is impossible. Finally, we get that $h\widetilde{b_n} \in (\overline{\mathcal{A}_n})_{\text{red}}$. Conversely, assume $h \in \mathrm{z.div}((\mathcal{B}_{n})_{\text{red}}/(\overline{\mathcal{A}_{n}})_{\text{red}})$. Then there exist $b_n \in \mathcal{B}_n$ with $b_n \not \in \overline{\mathcal{A}_n}$, $a \in \overline{\mathcal{A}_n}$ and a nilpotent $b$ from $\mathcal{B}$ such that
$hb_n=a+b$. But $a$ and $b$ are integral over $\mathcal{A}$. Then so is $hb_n$. Hence $hb_n \in \overline{\mathcal{A}_n}$, i.e. $h \in \mathrm{z.div}(\mathcal{B}_{n}/\overline{\mathcal{A}_{n}})$.

Consider $\rm(3)$. The statement follows from  $\rm(2)$ applied for each of the generators of $\mathfrak{m}_{\mathrm{red}}$ along with the fact that if $h$ is nilpotent in $\mathcal{R}$, then $h\mathcal{B}_n \subset \overline{\mathcal{A}_n}$ for each $n$.  
\end{proof}

The next lemma is key: it allows us to make the transition between the category of arbitrary local Noetherian schemes and that of local universally catenary Nagata schemes. The important observation is that associated points and most geometric properties of the objects considered behave well under completion. For one thing, a complete local ring  is universally catenary and Nagata. In particular, its integral closure in a reduced ring extension is module finite. Also, the  associated points of the quotients $\mathcal{N}^{n}/\overline{\mathcal{M}^{n}}$ can be detected from those of the corresponding  completions.

Let $X$ be a local Noetherian scheme, and $x_0$ its closed point. Denote by $\widehat{X}$ the completion of $X$ along $x_0$, and its closed point by $\widehat{x_0}$. Denote by $\mathfrak{m}_{x_0}$ and $\mathfrak{m}_{\widehat{x_0}}$ the maximal ideals of $\mathcal{O}_{X,x_0}$ and $\mathcal{O}_{\widehat{X,x_0}}$. For each $n$, denote by $\widehat{\mathcal{M}^{n}}$ and $\widehat{\mathcal{N}^{n}}$ the corresponding completions of $\mathcal{M}^{n}$ and $\mathcal{N}^{n}$. Finally, denote by $\overline{\widehat{\mathcal{M}^{n}}}$ the integral closure of $\widehat{\mathcal{M}^{n}}$ in $\widehat{\mathcal{N}^{n}}$.

Set $\widehat{B} := \mathrm{Proj}(\mathcal{R}(\mathcal{\widehat{M}}))$. Denote by $\widehat{b}$ the structure morphism from $\widehat{B}$ to $\widehat{X}$, and denote by $\widehat{B}_{i}$ the union of irreducible components of $\widehat{B}$ that surject onto $\widehat{X_i}$.
\begin{lemma}\label{key lemma} The following hold:
\begin{enumerate}
    \item[{\rm(1)}]  Let $h$ be an element from $\mathfrak{m}_{x_0}$. If $h \notin \mathrm{z.div}(\widehat{\mathcal{N}^{n}}/\overline{\widehat{\mathcal{M}^{n}}})$, then
    $h \not \in \mathrm{z.div}(\mathcal{N}^{n}/\overline{\mathcal{M}^{n}})$.
    \item[{\rm(2)}] 
         We have $\dim B_i = \dim \widehat{B}_i$. If $X$ is universally catenary, then  $\widehat{B}_i$  is equidimensional. Also, the fibers of $\mathrm{Proj}(\mathcal{R}(\widehat{\mathcal{M}}))$ and $\mathrm{Proj}(\mathcal{R}(\mathcal{M}))$ over $x_0$ are the same. Furthermore, if $X$ is universally catenary, then $\codim D_j$ in $\mathrm{Proj}(\mathcal{R}(\mathcal{M}))$ is the same as $\codim D_j$ in $\mathrm{Proj}(\mathcal{R}(\widehat{\mathcal{M}}))$ for each $j=1, \ldots, r.$
\end{enumerate}
\end{lemma}
\begin{proof} Consider $\rm{(1)}$. Identify $\mathcal{M}^{n}$, $\overline{\mathcal{M}^{n}}$ and $\mathcal{N}^{n}$ with their images in $\widehat{\mathcal{N}^{n}}$. By faithful flatness we have 
\begin{equation}\label{int. contraction}
\overline{\widehat{\mathcal{M}^{n}}} \cap \mathcal{N}^{n} = \overline{\mathcal{M}^{n}}.
\end{equation}
Indeed, $\mathcal{R}(\mathcal{N}) \otimes_{\mathcal{O}_{X,x_0}}
O_{\widehat{X,x_0}}$ is faithfully flat extension of $\mathcal{R}(\mathcal{N})$ because faithful flatness commutes with base change. Then (\ref{int. contraction}) follows at once by Prp.\ 1.6.2 from \cite{Huneke} applied with $R=\mathcal{R}(\mathcal{N})$, $S=\mathcal{R}(\mathcal{N}) \otimes_{\mathcal{O}_{X,x_0}}
O_{\widehat{X,x_0}}$ and $I:=\mathcal{M}^{n}\mathcal{R}(\mathcal{N})$. Next, suppose there exists $q \in \mathcal{N}^{n}$ such that $hq \in \overline{\mathcal{M}^{n}}$ and $q \notin \overline{\mathcal{M}^{n}}$. Then $hq \in \overline{\widehat{\mathcal{M}^{n}}}$ because by persistence of integral closure $\overline{\mathcal{M}^{n}} \subset \overline{\widehat{\mathcal{M}^{n}}}$.
But $h \notin \mathrm{z.div}(\widehat{\mathcal{N}^{n}}/\overline{\widehat{\mathcal{M}^{n}}})$ by assumption. Therefore, $q \in \overline{\widehat{\mathcal{M}^{n}}}$. By (\ref{int. contraction}) it follows that $q \in \overline{\mathcal{M}^{n}}$ which is contradiction. 

Consider $\rm{(2)}$. Let $\widehat{\eta_i}$ be the generic point of an irreducible component of $\widehat{X_i}$. Because $\widehat{X_i} \rightarrow X_i$ is faithfully flat, then the completion morphism maps $\widehat{\eta_i}$ to the generic point $\eta_i$ of an irreducible component $X_i$ of $X$. We have the following product diagram

\begin{displaymath}
\begin{CD}
\mathrm{Proj}(\mathcal{R}(\widehat{\mathcal{M}})) @>>> \widehat{X}\\
@VVV  @VVV\\
 \mathrm{Proj}(\mathcal{R}(\mathcal{M}))  @>>> X
\end{CD}
\end{displaymath}
where the horizontal arrows represent morphisms that surject irreducible components onto irreducible components and the right vertical arrow represents the completion morphism that surjects each irreducible component of $\widehat{X}$ to an irreducible component of $X$. Thus the left vertical map carries each irreducible component of $\mathrm{Proj}(\mathcal{R}(\widehat{\mathcal{M}}))$  to an irreducible component of $\mathrm{Proj}(\mathcal{R}(\mathcal{M}))$. 

Because 
$$\mathcal{O}_{X,\eta_i}^{e_i(n)} \otimes_{\mathcal{O}_{X,x_0}} \mathcal{O}_{\widehat{X},\widehat{\eta_i}}= \mathcal{O}_{\widehat{X},\widehat{\eta_i}}^{e_i(n)} \  \text{and} \  \widehat{\mathcal{M}^{n}}_{\widehat{\eta_i}}=(\mathcal{M}^{n}\otimes_{\mathcal{O}_{X,x_0}} \mathcal{O}_{X,\eta_i})\otimes_{\mathcal{O}_{X,\eta_i}}\mathcal{O}_{\widehat{X},\widehat{\eta_i}} $$ 
we get
$$\widehat{\mathcal{M}^{n}}_{\widehat{\eta_i}}=\mathcal{O}_{\widehat{X},\widehat{\eta_i}}^{e_{i}(n)}.$$
where $e_{i}(n)$ is the rank of $\mathcal{M}^{n}$ at $\eta_i$.
Applying Proposition \ref{structure morphism} \rm{(2)} for an irreducible component of $\widehat{X_i}$ of dimension equal to that of $X_i$ we get $\dim B_i = \dim \widehat{B}_i$. When $X$ is universally catenary, then $\widehat{X_i}$ is equidimensional by Thm.\ 31.7 in \cite{Matsumura}, so $\widehat{B}_i$ is equidimensional, too. 

Next we claim that
\begin{equation}\label{completion-Rees}
\widehat{\mathcal{M}}^{n}=\widehat{\mathcal{M}^{n}}.
\end{equation}
Indeed, by  Prp.\ 1.3 in \cite{Eisenbud}, the formation of the Rees algebra commutes with flat base change; so $\mathcal{R}(\mathcal{\widehat{M}}) = \mathcal{R}(\mathcal{M}) \otimes_{\mathcal{O}_{X,x_0}} \mathcal{O}_{\widehat{X,x_0}}$.
Thus (\ref{completion-Rees}) holds. 

Next, we have
$$\widehat{\mathcal{M}}^{n}/\mathfrak{m}_{\widehat{x_0}}\widehat{\mathcal{M}}^{n}=\widehat{\mathcal{M}^{n}}/\widehat{\mathfrak{m}_{x_0}\mathcal{M}^{n}}= \mathcal{M}^{n}/\mathfrak{m}_{x_0}\mathcal{M}^{n}$$
which proves that the fibers of $\mathrm{Proj}(\mathcal{R}(\mathcal{M}))$ and $\mathrm{Proj}(\mathcal{R}(\widehat{\mathcal{M}}))$ over $x_0$ and $\widehat{x_0}$ are the same.

Assume $X$ is universally catenary and $D_j$ is contained in $B_i$ such that $\codim (D_j,B_i) = \codim D_j$. Then $D_j$ is contained in $\widehat{B}_i$. But $\widehat{B}_i$ is equidimensional of the same dimension as $B_i$. Hence by Prp.\  \ref{structure morphism} {\rm(3)} the codimension of $D_j$ in $\mathrm{Proj}(\mathcal{R}(\mathcal{M}))$ is the same as the codimension of $D_j$ in $\mathrm{Proj}(\mathcal{R}(\widehat{\mathcal{M}})).$
\end{proof}
Note that if $\overline{\mathcal{M}^{n}} = \overline{\widehat{\mathcal{M}^{n}}}$ were true, then the converse of $\rm(1)$  would follow at once. However, a necessary condition for the last identity to be valid is that $X$ is analytically unramified (see Chp. 9 in \cite{Huneke}). Finally, we would like to remark that the universally catenary assumption is needed to guarantee that $\widehat{X_i}$ is equidimensional which in turn insures that the codimension of $D_j$ remains the same. 

Before stating our main result let's fix some notation. Assume $X$ is local with closed point $x_0$. As usual for each $h \in \mathfrak{m}_{x_0}$ denote by $H$ the subscheme of zeroes of $h$ in $X$. Let $S$ be the union of the nonfree loci of $\mathcal{M}$ and $\mathcal{N}$ and $Z$ be the nonfree locus of $\mathcal{N}/\mathcal{M}$. If we work with the hypothesis that $\mathcal{M}$ and $\mathcal{N}$ are contained in a free module $\mathcal{F}$, then replace $X$ by $X_{\mathrm{red}}$ and $\mathcal{M}$ and $\mathcal{N}$ with their images $\mathcal{F}\otimes_{\mathcal{O}_{X}} \mathcal{O}_{{X}_{\mathrm{red}}}$, and define $Z$ and $S$ accordingly. 

Recall that $B:=\mathrm{Proj}(\mathcal{R}(\mathcal{M}))$,  $b$ is the structure morphism $b \colon B \rightarrow X$, and $B_i:=b^{-1}X_{i}$. Denote by $b_i$ the restriction of $b$ to $B_i$. Let $S_{i}(k)$ be the closed subschemes of $X$ consisting of all points $x \in X_{i}$ such that $b_{i}^{-1}x \geq k$. Denote by $D_1, \ldots, D_q$ the irreducible components of $b^{-1}x_0$.
\begin{theorem}\label{conormal} Suppose $X$ is universally catenary. For each $j=1, \ldots, q$ assume that
    $$\codim D_j \geq 2.$$
Let $h \in \mathfrak{m}_{x_0}$ such that $H$ cuts properly the irreducible  components of positive dimension of $Z$ and $S_{i}(k)$ for each $k$. Then  $h \not \in \mathrm{z.div}(\mathcal{N}^{n}/\overline{\mathcal{M}^{n}})$ for each $n$.
\end{theorem}
\begin{proof} Suppose $\dim X=0$. Then $\dim b^{-1}x_0 = \dim B$. Hence $\codim D_j =0$ in $B$ for each $j$ which contradicts with our assumptions. Suppose $\dim X = 1$. Because $h$ avoids the minimal primes of $\mathcal{O}_{X,x_0}$, then $H$ cuts the irreducible components of $B$ properly. Hence each irreducible component of $b^{-1}H$ is of codimension one in an irreducible component of $B$. On the other hand, $\dim b^{-1} H = \dim b^{-1} x_0$ because $H$ is supported at $x_0$. Thus for each $D_j$ we have $\codim D_j=1$. Finally, we conclude that the assumption $\codim D_j \geq 2$ implies $\dim X_i \geq 2$ for some $i$. 

Replace $\mathcal{R}(\mathcal{M})$ by its image in $\mathcal{R}(\mathcal{N})$. By Prp.\ \ref{structure morphism} \rm{(4)} our hypothesis remain intact. Let us fix some notation. Set  
$$\mathcal{R}^{\dagger}(\mathcal{M}) := \mathcal{O}_{X,x_0} \oplus \overline{\mathcal{M}} \oplus \overline{\mathcal{M}^{2}} \oplus \cdots.$$
Note that $\mathcal{R}^{\dagger}(\mathcal{M})$ is the integral closure of $\mathcal{R}(\mathcal{M})$ in $\mathcal{R}(\mathcal{N})$ (compare with Prp. 5.2.1 from \cite{Huneke}). The next part of the proof involves several reduction steps that will allow us to 
assume that $X$ is local reduced Nagata universally catenary scheme. 

First, complete $X$ along $x_0$ and consider $\mathcal{R}(\mathcal{M}) \otimes_{\mathcal{O}_{X,x_0}}
O_{\widehat{X,x_0}}$ and $\mathcal{R}(\mathcal{N}) \otimes_{\mathcal{O}_{X,x_0}}
O_{\widehat{X,x_0}}$. By Lem.\ \ref{key lemma} \rm{(2)} and our assumptions we have $\codim D_j \ge 2$ in $\mathrm{Proj}(\mathcal{R}(\widehat{\mathcal{M}}))$ for each $j=1, \ldots, r.$ By flatness the image of $\widehat{\mathcal{M}^{n}}$ in $\mathcal{R}(\mathcal{\widehat{N}})$ is the same as the completion of the image of $\mathcal{M}^{n}$ in $\mathcal{R}(\mathcal{N})$. By Lem.\ \ref{key lemma} \rm{(1)} it's enough to show that $h \notin \mathrm{z.div}(\widehat{\mathcal{N}^{n}}/\overline{\widehat{\mathcal{M}^{n}}})$. Therefore, we can assume that $X$ is complete local scheme and and hence local universally catenary Nagata scheme, and assume that $\mathcal{M}^{n}$ and $\mathcal{N}^n$ are complete with respect to $x_0$, too.  

Next, by Lem.\ \ref{reduction} we can replace $X$, $\mathcal{R}(\mathcal{M})$, $\mathcal{R}^{\dagger}(\mathcal{M})$ and $\mathcal{R}(\mathcal{N})$ with their reduced structures. As $\mathcal{R}(\mathcal{M})$ is of finite type over a Nagata ring, and $\mathcal{R}(\mathcal{N})$ is reduced, then $\mathcal{R}^{\dagger}(\mathcal{M})$ is module-finite over $\mathcal{R}(\mathcal{M})$ by  \cite[\href{http://stacks.math.columbia.edu/tag/03GH}{Tag 03GH}]{Stacks} (cf.\ Ex.\ 9.7 in \cite{Huneke}). 

Assume $x$ is not in $Z$ and $S$. Because the formation of integral closure and localization commute we have $\overline{\mathcal{M}^{n}}_{x} = \overline{\mathcal{M}^{n}_{x}}$. By Prp.\ 1.3 \cite{Eisenbud} $\mathcal{R}(\mathcal{M})_{x}$ and $\mathcal{R}(\mathcal{N})_{x}$ are polynomial rings (see the proof of Prp.\ \ref{support}). Because $X$ is reduced and $\mathcal{M}_x$ is a direct summand of $\mathcal{N}_x$, then
$\mathcal{R}(\mathcal{M})_x$ is integrally closed in $\mathcal{R}(\mathcal{N})_x$. Hence $\overline{\mathcal{M}^{n}_{x}}=\mathcal{M}^{n}_x$. Finally,
\begin{equation}\label{int. clos.}
\overline{\mathcal{M}^{n}}_{x} = \mathcal{M}^{n}_{x}.
\end{equation}

Set $P:=\mathrm{Proj}(\mathcal{R}^{\dagger}(\mathcal{M})).$ Denote by $b^{\dagger}: P \rightarrow B$ the finite surjective map induced by the inclusion $\mathcal{R}(\mathcal{M}) \hookrightarrow \mathcal{R}^{\dagger}(\mathcal{M})$. Because $Z$ and $S$ do not contain the generic points of the irreducible components of $X$, then by (\ref{int. clos.}) it follows that 
$\mathcal{R}(\mathcal{M})$ and $\mathcal{R}^{\dagger}(\mathcal{M})$ are generically equal. So $b^{\dagger}$ is finite and birational. Set $f=b\circ b^{\dagger}$. For each irreducible component $X_i$ of $X$ set $P_i = f^{-1}X_i$. Because $B_i$ is irreducible and $b^{\dagger}$ is birational it follows that $P_i$ is irreducible of dimension equal to that of $B_i$. 


Let $D^{\dagger}$ be an irreducible component of $f^{-1}x_0$. Then $\dim D^{\dagger} = \dim b^{\dagger}(D^{\dagger})$ because $f^{-1}x_0 \rightarrow b^{-1}x_0$ is finite. Suppose $b^{\dagger}(D^{\dagger}) \subset B_i$ for some $i$ such that $\codim (b^{\dagger}(D^{\dagger}),B_i)=\codim b^{\dagger}(D^{\dagger})$. Because $D^{\dagger} \subset P_i$,  $\dim P_i = \dim B_i$ and $\codim D_j \geq 2$ for each $j$ we have
\begin{equation}\label{codimension}
\codim D^{\dagger} \geq 2.
\end{equation}
Select $H$ as in Thm.\ \ref{vertical} for $f\colon P \rightarrow X$. Then Thm.\ \ref{vertical} implies that the horizontal components of $f^{-1}H$ surject onto those of $H$. But (\ref{codimension}) implies that $f^{-1}H$ does not have vertical components. Hence all irreducible components of $f^{-1}H$ surject onto those of $H$.

Further restrict $H$ so that the irreducible components of $H$ are not contained in those of $Z$ and $S$. The last we can do because $\dim X_i \ge 2$ for some $i$ as discussed in the proof of Prp.\ \ref{support}.

Consider the homomorphism $$\phi_{h}: \mathcal{R}^{\dagger}(\mathcal{M})/h\mathcal{R}^{\dagger}(\mathcal{M}) \longrightarrow \mathcal{R}(\mathcal{N})/h\mathcal{R}(\mathcal{N}).$$
Denote its kernel by $I_{\phi_{h}}$. We identify
\begin{equation}
I_{\phi_{h}} = (h \mathcal{R}(\mathcal{N}) \cap \mathcal{R}^{\dagger}(\mathcal{M}))/h\mathcal{R}^{\dagger}(\mathcal{M}).
\end{equation}


Next, by  (\ref{int. clos.}) and Prp. \ref{support} applied to $\overline{\mathcal{M}^{n}}$ for each $n$, we obtain that each component of the support of $I_{\phi_{h}}$ in $X$ as an $\mathcal{O}_{X,x_0}$-module is properly contained in an irreducible component of $H$.
Therefore, $I_{\phi_{h}}$ now viewed as $\mathcal{O}_{f^{-1}H}$-module, vanishes locally at the horizontal components of $f^{-1}H$ because the irreducible components of $f^{-1}H$ project in $X$ onto those of $H$.  However, as we already determined $f^{-1}H$ does not have vertical components. So, each component of the support of the $\mathcal{O}_{f^{-1}H}$-module $I_{\phi_{h}}$ is properly contained in an irreducible component of $f^{-1}H$. But $I_{\phi_{h}}$ is also an ideal in $\mathcal{R}^{\dagger}(\mathcal{M})/h\mathcal{R}^{\dagger}(\mathcal{M})$. 
Thus, $\mathrm{Ass}_{f^{-1}H}(I_{\phi_{h}})$ consists of the generic points of embedded components of $f^{-1}H$.

Let $p \in \mathrm{Ass}_{f^{-1}H}(I_{\phi_{h}})$ and let $\mathfrak{m}_{p}$ be the ideal of $p$ in $\mathcal{R}^{\dagger}(\mathcal{M})$. Then  $\mathfrak{m}_{p}$ is an associated prime of $h\mathcal{R}^{\dagger}(\mathcal{M})$. Thus we can write
\begin{equation}\label{sat}
\mathfrak{m}_p=(h\mathcal{R}^{\dagger}(\mathcal{M}) :_{\mathcal{R}^{\dagger}(\mathcal{M})} y)
\end{equation}
for some $y \in \mathcal{R}^{\dagger}(\mathcal{M})$ that maps to a nonzero element of $I_{\phi_{h}}$ under $\phi_{h}$.
Note that $y/h \in \mathcal{R}(\mathcal{N})$, because $y \in h\mathcal{R}(\mathcal{N}) \cap \mathcal{R}^{\dagger}(\mathcal{M})$. By (\ref{sat}) we have
$(y/h)\mathfrak{m}_{p} \subset \mathcal{R}^{\dagger}(\mathcal{M})$. 

Suppose $(y/h)\mathfrak{m}_p \subset \mathfrak{m}_p$. Note that $h$ is a nonzero divisor of $\mathcal{R}(\mathcal{N})$ because the latter ring is reduced and $h$ avoids the minimal primes of $X$ and hence those of $\mathcal{R}(\mathcal{N})$. Because $\mathfrak{m}_p$ contains $h$, then $\mathfrak{m}_p$ is a faithful $\mathcal{R}^{\dagger}(\mathcal{M})[y/h]$-module.
Then by the Determinantal Trick Lemma (see Lem.\ 2.1.8 from \cite{Huneke}) it follows that $y/h$ is integral over $\mathcal{R}^{\dagger}(\mathcal{M})$. But $\mathcal{R}^{\dagger}(\mathcal{M})$ is integrally closed in
$\mathcal{R}(\mathcal{N})$. Hence $y/h \in \mathcal{R}^{\dagger}(\mathcal{M})$, i.e.\ $y \in  h\mathcal{R}^{\dagger}(\mathcal{M})$ which is contradiction with our choice of $y$. 

Then $(y/h)\mathfrak{m}_p \not \subset \mathfrak{m}_p$. So, locally at $p$ there exists $z \in \mathcal{O}_{P,p}$ such that $(y/h)z=1$. Therefore, by (\ref{sat}) we get $\mathfrak{m}_p=(z)$ in $\mathcal{O}_{P,p}$.  Hence $\mathfrak{m}_p$ is a minimal prime of $h\mathcal{R}^{\dagger}(\mathcal{M})$, or in other words $\overline{\{p\}}$ is an irreducible component of $f^{-1}H$. But this is impossible because by assumption $\overline{\{p\}}$ is an embedded component of $f^{-1}H$. Therefore, $I_{\phi_{h}}=0$. But $h$ is a nonzero divisor of $\mathcal{R}(\mathcal{N})$ so $h \not \in \mathrm{z.div}(\mathcal{N}^{n}/\overline{\mathcal{M}^{n}})$ for each $n$.
\end{proof} 
The following immediate corollary of Thm.\ \ref{conormal} is a far reaching generalization of  results of McAdam \cite{McAdam} for ideals,  Rees \cite{Rees}, Katz and Rice \cite{Katz2} for modules $\mathcal{M}$ contained in a free module $\mathcal{F}$ such that $e_i = \mathrm{rk}(\mathcal{F})$ for each $i=1, \ldots, r$. 

\begin{corollary}\label{McAdam}
Assume $X$ is a universally catenary local scheme with closed point $x_0$. If $x_0 \in \mathrm{Ass}_{X}(\mathcal{N}^n/\overline{\mathcal{M}^n})$ for some $n$, then $\codim D_j \leq 1$ for some $j$.
\end{corollary}
The final part of the proof of Thm.\ \ref{conormal} follows closely the argument used to show that there are no embedded primes
in the primary decomposition of a principal ideal in a ring integrally closed in its total ring of fractions (compare with Prp.\ 4.1.1 in \cite{Huneke}).

In the setup of Thm.\ \ref{conormal} assume  $X$ is normal and assume $\mathcal{N}$ is a free module $\mathcal{F}$. Then $\mathcal{R}(\mathcal{F})$ is the symmetric algebra of  $\mathcal{F}$ which is integrally closed because it is a polynomial ring over an integrally closed ring $\mathcal{O}_{X,x_0}$. Therefore, $\mathcal{R}^{\dagger}(\mathcal{M})$ is the integral closure of $\mathcal{R}(\mathcal{M})$ in its total ring of fractions
(compare with Prp.\ 5.2.4 in \cite{Huneke}). In particular, it satisfies Serre's condition $S_2$. In this case, we see easily that since $P$ is $S_2$, then $f^{-1}H$ has no embedded components. Therefore, $\mathrm{Ass}_{f^{-1}H}(I_{\phi_{h}}) = \emptyset$, so again $I_{\phi_{h}}=0$.

If $X$ is normal and $S$ is of dimension at most one, then we can avoid Bertini's theorem for extreme morphisms. Indeed, by Prp.\ \ref{support} we get directly that $I_{\phi_{h}}$ vanishes locally at the irreducible components of $f^{-1}H$. Additionally, if $\dim X=2$, then $X$ is Cohen--Macaulay and thus universally catenary, and $\dim S \leq 1$. In this situation our Cor.\ \ref{McAdam}  and Cor.\ \ref{converse deficient} generalize Cor.\ $7$ in \cite{McAdam}.

The results from the previous two sections allow us to specify explicitly the genericity conditions on $h$. Moreover, as explained in Rmk.\ \ref{dense}, if we assume that $X$ is of finite type over an infinite field $\Bbbk$, then we can construct a Zariski open dense subset $U_h$ in $\mathbb{A}_{\Bbbk}^{N}$ of elements $h$ having this property, where $N:=\dim_{\Bbbk}(\mathfrak{m}_{x_0}^{2}/\mathfrak{m}_{x_0})$.

Finally, note that under the assumptions imposed on $\mathcal{M}$ and $\mathcal{N}$ in the beginning of the section, we observed in Prp.\ \ref{structure morphism} $\rm{(4)}$ that the kernel of the homomorphism $\mathcal{R}(\mathcal{M}) \rightarrow \mathcal{R}(\mathcal{N})$ is nilpotent. This observation shows that the codimension hypothesis on each $D_j$ is preserved after replacing  $\mathcal{R}(\mathcal{M})$ with its image in  $\mathcal{R}(\mathcal{N})$. Moreover, the $S_i(k)$ too remain unaffected after the replacement. However, $Z$ depends on the embedding $\mathcal{M} \hookrightarrow \mathcal{N}$ and so does our choice of $h$. In general, without additional hypothesis on $\mathcal{M}$ and $\mathcal{N}$ one can construct examples such that $\codim D_j$ drops to  $1$ after replacing $\mathcal{R}(\mathcal{M})$ with its image in $\mathcal{R}(\mathcal{N})$. 

\begin{remark}\label{general main} \rm One can prove a more general version of Thm.\ \ref{conormal} if instead of Rees algebras we work with a pair $\mathcal{A} \subset \mathcal{B}$ of standard graded finitely generated $R$-algebras. The only obstacle to extending our results to this setting is proving an analogue of Prp.\ \ref{support}. Here is how it can be done. By Thm.\ \ref{finite ass pnts} the set $\bigcup_{n=1}^{\infty}\mathrm{Ass}_{X}(\mathcal{B}_n/\mathcal{A}_n)$ is finite. Select $h$ such that $H$ cuts properly the closures of those associated points which are of positive dimension. Further, restrict $h$ so that in addition to the usual assumptions $\widehat{H}$, the scheme of zeroes defined by the image of $h$ in the reduction of $\mathcal{O}_{\widehat{X,x_0}}$, cuts properly the components of  $\mathrm{Supp}_{\widehat{X,x_0}}(\mathcal{A}^{\dagger}/\mathcal{A})$ where $\mathcal{A}^{\dagger}$ is the integral closure of $\mathcal{A}$ in $\mathcal{B}$. With this choice of $h$, the proof of Thm.\ \ref{conormal} extends to this more general setting. The disadvantage of working in such generality though, is that we are no longer able to specify explicitly the genericity conditions on $h$. 
\end{remark}

\section{Generalized Kleiman--Thorup Theorem}\label{GKT}
As a first application of our main theorem we strengthen a fundamental result of Kleiman and Thorup (cf. Thm.\ A.1 from \cite{Thorup}) which we derive as a corollary to Thm.\ \ref{conormal}. Let $X$ be a local Noetherian universally catenary scheme of positive dimension. Consider the pair $\mathcal{M} \subset \mathcal{N}$ of $R$-modules that are free of rank $e_i$ at the generic point of each irreducible component $X_i$ of $X$. Set $B:=\mathrm{Proj}(\mathcal{R}(\mathcal{M}))$. Let $b\colon B \rightarrow X$ be the structure morphism. Denote by $B_{i}$ the irreducible component of $B$ that surjects onto $X_i$ and by $b_{i}$ the restriction of $b$ to $B_i$. 
Let $Y$ be the closed set in $X$ where $\mathcal{N}$ is not integral over $\mathcal{M}$.
Denote by $\overline{\mathcal{M}}$ the integral closure of $\mathcal{M}$ in $\mathcal{N}$.
\begin{theorem}[Kleiman--Thorup]\label{Kleiman-Thorup}
Assume $\mathcal{N}$ is not integral over $\mathcal{M}$. For each irreducible component $Y_j$  of $Y$ there exists an irreducible component $X_i$ of $X$ that contains $Y_j$ and $$\codim b_{i}^{-1}Y_j = 1$$ in $B_i$.
\end{theorem}
\begin{proof} Let $Y_j$ be an irreducible component of $Y$ and denote by $x_0$ be its generic point.  
Set $X':=\mathrm{Spec}(\mathcal{O}_{X,x_0})$. Because $\mathcal{N}$ and $\mathcal{M}$ are generically equal as shown in the proof of Prp.\ \ref{support}, then $Y_j$ is properly contained in the irreducible components of $X$. Thus $\dim X' \ge 1$. Because $Y_j$ is a component of $\mathrm{Supp}_{X}(\mathcal{N}/\overline{\mathcal{M}})$, then $x_0 \in \mathrm{Ass}_{X}(\mathcal{N}/\overline{\mathcal{M}})$ which in turn implies that $x_0 \in \mathrm{Ass}_{X'}(\mathcal{N}_{x_0}/\overline{\mathcal{M}}_{x_0})$. Since formation of integral closure and localization commute, it follows that $\overline{\mathcal{M}}_{x_0}=\overline{\mathcal{M}_{x_0}}$. Thus, $x_0 \in \mathrm{Ass}_{X'}(\mathcal{N}_{x_0}/\overline{\mathcal{M}_{x_0}})$. Therefore, by Cor.\ \ref{McAdam}
there exists a component $D_j$ of $b^{-1}x_0$ such that $\codim D_j = 1$
in an irreducible component of $\mathrm{Proj}(\mathcal{R}(\mathcal{M}_{x_0}))$. 

Suppose that this component of $\mathrm{Proj}(\mathcal{R}(\mathcal{M}_{x_0}))$ surjects onto an irreducible component $X_{i}'$ of $X'$ which maps in turn to an irreducible component $X_i$ of $X$. Then $D_j \subset B_i$ and
\begin{equation}\label{K-T1}
\dim D_j = \dim X_{i}'+e_i-2
\end{equation}
by Prp.\ \ref{structure morphism} \rm{(2)}. Because $D_j \subset b^{-1}x_0$ and $D_j \subset B_i$, it follows that $D_j \subset b_{i}^{-1}x_0$. Let $W_j$ be a component of $b_{i}^{-1}Y_j$ that contains $D_j$. Because $b_i$ maps $D_j$ onto $x_0$, then $b_i$ maps $W_j$ onto $Y_j$. The dimension formula (\cite[\href{http://stacks.math.columbia.edu/tag/02JX}{Tag 02JX}]{Stacks} or Lem.\ 3.1 (ii) in \cite{KT-Al}) applied for the  proper map $b_i$ along with the inequality $\dim b_{i}^{-1}x_0 \geq \dim D_j$ yield
\begin{equation}\label{K-T2}
\dim W_j \geq \dim Y_j + \dim D_j.
\end{equation}
Because $X$ is local and catenary, then so is $X_i$. Thus
\begin{equation}\label{K-T3}
\dim X_{i}' = \dim X_{i} - \dim Y_j.
\end{equation}
Combining (\ref{K-T1}), (\ref{K-T2}) and (\ref{K-T3}) we get
$$\dim W_j \geq \dim X_{i} +e_{i}-2.$$
Because the components of $b^{-1}Y$ are properly contained in those of $B$ we have $$\dim b_{i}^{-1}Y_j \leq \dim X_i +e_i-2.$$ Finally, because $W_j \subset b_{i}^{-1}Y_{j}$ we obtain $\dim b_{i}^{-1}Y_j = \dim X_i+e_i-2$, or equivalently $$\codim b_{i}^{-1}Y_j =1$$ in $B_i$ as desired.
\end{proof}
The original version of Thm.\  \ref{Kleiman-Thorup} is stated in \cite{KT-Al} (see \cite{KT} and \cite{Thorup} for subsequent versions with slightly different hypothesis). In \cite{KT} it's assumed that $X$ is irreducible universally catenary scheme, $\mathcal{N}$ is torsion-free and $\mathcal{M}$ and $\mathcal{N}$  are generically equal. In \cite{Thorup} $X$ is assumed  to be equidimensional and reduced. The modules $\mathcal{M}$ and $\mathcal{N}$ are assumed to be free of constant rank at $e$ at the generic point of each component of $X$. Additionally, $\mathcal{M}$ and $\mathcal{N}$ are contained in a free module. In our treatment we allow the irreducible components of $X$ to be of different dimensions and the ranks of $\mathcal{M}$ and $\mathcal{N}$ to vary along the generic  points of the irreducible components of $X$. Alternatively, tracing through the proof of Thm.\ \ref{conormal} and Prp.\ \ref{support} we see that we can replace the rank assumption on $\mathcal{M}$ and $\mathcal{N}$ with the hypothesis that the two modules are equal at the generic point of each irreducible component of $X$. Finally, in \cite{KT}  and \cite{Thorup} it's proved that $b^{-1}Y$ is of codimension one, whereas we prove that $b_{i}^{-1}Y_j$ is of codimension one in an irreducible component of $B$ for each irreducible component $Y_j$ of $Y$.

\section{Finiteness and asymptotic stability of associated points}\label{finiteness and asymptotic stability}

Let $X:=\mathrm{Spec}(R)$ be an affine Noetherian scheme and let $\mathcal{M} \subset \mathcal{N}$ be finitely generated $R$-modules. In this section we prove that $\mathrm{Ass}_{X}(\mathcal{N}^{n}/\overline{\mathcal{M}^{n}})$ and $\mathrm{Ass}_{X}(\mathcal{N}^{n}/\mathcal{M}^{n})$ are asymptotically stable. We derive these results as special cases of more general results (see Prps.\ \ref{finite ass pnts} and \ref{asymptotics algebras} and Thm.\ \ref{int. asym.}) about finitely generated graded algebras. Prps.\ \ref{finite ass pnts} and \ref{asymptotics algebras} are results by Katz and Puthenpurakal (see Lem.\ 3.1 in \cite{Katz3}) and Hayasaka (see Thm.\ 1.1 in \cite{Hayasaka}). We present our own proofs of these two propositions here because they are self-contained and simple. Our  Thm.\ \ref{int. asym.} generalizes results by Rees \cite{Rees81}, Ratliff \cite{Ratliff84} in the ideal case (see Sct.\ 10 in \cite{Sharp} for a recent historical account of these results), and results by Katz and Naude \cite{Katz1} for modules admitting embeddings in free modules. 

If  $X$ is universally catenary and  $\mathcal{M}$ and $\mathcal{N}$ satisfy the hypothesis from the beginning of Sct.\ \ref{main results}, we give a comprehensive geometric description of  $\mathrm{Ass}_{X}(\mathcal{N}^{n}/\overline{\mathcal{M}^{n}})$: if $x \in \mathrm{Ass}_{X}(\mathcal{N}^{n}/\overline{\mathcal{M}^{n}})$, then either $x$ is the generic point of a codimension-one component of the nonfree locus of $\mathcal{N}/\mathcal{M}$, or  $x$ is the generic point of an irreducible component of some $S_{i}(k)$ as defined in (\ref{S(i)}). Under the universally catenary hypothesis, applying Chevalley's constructability result  
we prove independently of Thm.\ \ref{int. asym.} that $\bigcup_{n=0}^{\infty}\mathrm{Ass}_{X}(\mathcal{N}^{n}/\overline{\mathcal{M}^{n}})$ is finite. 

As demonstrated in the proof of the next proposition (see Lem.\ 3.1 in \cite{Katz3}), showing that  $\bigcup_{n=0}^{\infty}\mathrm{Ass}_{X}(\mathcal{N}^{n}/\mathcal{M}^{n})$ is finite is fairly trivial. The reason is that the modules $\mathcal{M}^{n}$ as $n$ varies form a finitely generated algebra over $R$. Our proof relies on a modification of Ratliff's argument for the ideal case (see the beginning of Thm.\ 2.11 from \cite{Ratliff}). 

Let $\mathcal{A} \subset \mathcal{B}$ be finitely generated $\mathbb{N}_0$-graded $R$-algebras. Denote their $n$th graded pieces by $\mathcal{A}_n$ and $\mathcal{B}_n$ respecitevely. Set $\mathcal{A}_{0}=\mathcal{B}_{0}=R$.

\begin{proposition}[Katz--Puthenpurakal]\label{finite ass pnts}
Assume $\mathcal{A}$ is standard graded. Then
$$\bigcup_{n=1}^{\infty}\mathrm{Ass}_{X}(\mathcal{B}_n/\mathcal{A}_n)$$
is finite.
\end{proposition}
\begin{proof} Let $a_1, \ldots, a_k$ be the generators of $\mathcal{A}$ over $R$. 
Let $t$ be an indeterminate and set $u:=1/t$. Consider the ring $$\mathcal{R}:=\mathcal{B}[a_1t,\ldots,a_kt,u].$$ It's a subring of $\mathcal{B}[t,u]$. Denote by  $\langle \mathcal{A}_n \rangle$ the ideal in $\mathcal{B}$ generated by $\mathcal{A}_n$. The elements in $\mathcal{R}$ are finite sums $\sum_{-p}^{q}c_{n}t^{n}$ where $c_n \in \langle \mathcal{A}_n \rangle$ with the convention that $\langle \mathcal{A}_n \rangle =\mathcal{B}$ for $n \leq 0$. Furthermore, $u$ is a regular element in $\mathcal{R}$ and $u^{n}\mathcal{R} \cap \mathcal{B} = \langle \mathcal{A}_n \rangle$. Because $u$ is regular, the associated primes of $u^n$ are the associated primes of $u$ for each $n$. Hence $\bigcup_{n=0}^{\infty}\mathrm{Ass}_{\mathcal{B}}(\mathcal{B}/\langle \mathcal{A}_n \rangle)$ is finite (see the beginning of the proof of Thm.\ 2.1. in \cite{Ratliff}). 

Next, let $x \in X$ be an associated point of the $R$-module $\mathcal{B}_n/\mathcal{A}_n$ for some $n$. Denote the ideal of $x$ in $R$ by $\mathcal{I}_x$ and let $\widetilde{b_n}$ be an element of the $\mathcal{B}_n/\mathcal{A}_n$ such that its annihilator is equal to $\mathcal{I}_x$. Let $\mathcal{I}(\widetilde{b_n})$ be the ideal in $\mathcal{B}$ that annihilates $\widetilde{b_n}$ as an element of $\mathcal{B}/\langle \mathcal{A}_n \rangle$. Then $\mathcal{I}(\widetilde{b_n}) \cap R = \mathcal{I}_x$. Because $\mathcal{I}_x$ is prime, the contraction of the radical  $\mathcal{I}(\widetilde{b_n})$ to $R$  is $\mathcal{I}_x$. However, the radical of $\mathcal{I}(\widetilde{b_n})$ is the intersection of minimal associated primes of the $\mathcal{B}$-module $\mathcal{B}/\langle \mathcal{A}_n \rangle$. Therefore, there exists an associated prime $Q(\widetilde{b_n})$ of $\mathcal{B}/\langle \mathcal{A}_n \rangle$ from $\mathcal{B}$ such that  its contraction to $R$ is $\mathcal{I}_x$. But $\bigcup_{n=0}^{\infty}\mathrm{Ass}_{\mathcal{B}}(\mathcal{B}/\langle \mathcal{A}_n \rangle)$ is finite as shown above. Therefore, there are finitely many $x$ from $X$ that are associated points of the quotients $\mathcal{B}_n/\mathcal{A}_n$.
\end{proof}
Suppose $\mathcal{M} \subset \mathcal{N}$ are finitely generated $R$-modules with no further hypothesis. Applying Prp.\ \ref{finite ass pnts} with  $\mathcal{A}$ equal to the image of $\mathcal{R}(\mathcal{M})$ in  $\mathcal{B}:=\mathcal{R}(\mathcal{N})$  we get the following corollary.
\begin{corollary}\label{finitely many}
The set $\bigcup_{i=1}^{\infty}\mathrm{Ass}_{X}(\mathcal{N}^{n}/\mathcal{M}^{n})$ is finite.
\end{corollary}

Assume $\mathcal{A} \subset \mathcal{B}$ are finitely generated standard graded $R$-algebras. The following result (see Thm.\ 1.1 in \cite{Hayasaka} and Cor.\ 3.3. in \cite{Katz3} for a generalization) shows that the sequence of sets $\mathrm{Ass}_{X}(\mathcal{B}_n/\mathcal{A}_n)$ as $n$ varies is asymptotically 
stable. 
\begin{proposition}[Hayasaka]\label{asymptotics algebras}
Suppose $x_0 \in \mathrm{Ass}_{X}(\mathcal{B}_n/\mathcal{A}_n)$ for infinitely  many $n$. Then $x_0 \in \mathrm{Ass}_{X}(\mathcal{B}_n/A_n)$ for each $n$ large enough. In particular, there exists $n_0$ such that
\begin{equation}\label{stable}
\mathrm{Ass}_X(\mathcal{B}_{n_{0}}/\mathcal{A}_{n_{0}})= \mathrm{Ass}_X(\mathcal{B}_{n_{0}+k}/\mathcal{A}_{n_{0}+k})
\end{equation}
for any  $k$.
\end{proposition}
\begin{proof}
First, $x_0 \in \mathrm{Ass}_{X}(\mathcal{B}_n/\mathcal{A}_n)$ if and only if $x_0$ is an associated point of $(\mathcal{B}_n)_{x_0}/(\mathcal{A}_n)_{x_0}$, so we can assume that $X$ is local with closed point $x_0$. Suppose $x_0 \notin \mathrm{Ass}_{X}(\mathcal{B}_n/\mathcal{A}_n)$ for infinitely many $n$. Because  $\bigcup_{i=1}^{\infty}\mathrm{Ass}_{X}(\mathcal{B}_n/\mathcal{A}_n)$ is finite by Prp.\ \ref{finite ass pnts}, then by prime avoidance there exists $h$ from the maximal ideal  $\mathfrak{m}_{x_0}$ such that $ h \notin \mathrm{z.div}(\mathcal{B}_n/\mathcal{A}_n)$ for infinitely many $n$. In other words, consider the map $$\phi_h \colon \mathcal{A}/h\mathcal{A} \longrightarrow \mathcal{B}/h\mathcal{B}$$ and denote its kernel by $K(h)$. Then $K(h)_n =0$ for infinitely many $n$.
Let $a_1, \ldots, a_s$ be generators of $\mathcal{A}/h\mathcal{A}$ as an $R$-algebra and let $k_1, \ldots, k_l$ be generators of $K(h)$ as an ideal in $\mathcal{A}/h\mathcal{A}$. 
Select $n_1$  large enough so that $K(h)_{n_1}=0$. Then $k_{j}a_{i}^{u_{j}(k_j)}=0$ where $u_{j}(k_j)$ is an integer depending on $k_j$ with $\deg (k_j)+u_{i}(k_j)=n_1$. Thus, $K(h)_n=0$ for all large $n$. Therefore, for all large $n$ if $h \in \mathrm{z.div}(\mathcal{B}_n/\mathcal{A}_n)$, then 
there exists $b_n \in \mathcal{B}_n$ with $b_n \not \in \mathcal{A}_n$ such that $hb_n=0$. 

Denote by $\mathcal{I}_{\mathcal{B}}$ and $\mathcal{I}_{\mathcal{A}}$ the homogeneous ideals in $\mathcal{B}$ and $\mathcal{A}$ respectively that are generated by the elements in the corresponding algebras annihilated by $h$. 
Let $f_1, \ldots, f_q$ be homogeneous generators of $\mathcal{I}_{\mathcal{B}}$ and let $b_1, \ldots, b_p$ be degree one  generators of $\mathcal{B}$ as an $R$-algebra. Select $n_1$ large enough such that $h \not \in \mathrm{z.div}(\mathcal{B}_{n_1}/\mathcal{A}_{n_1})$ and $n_1 > \deg{f}_i$ for each $i=1, \ldots, q$. Then for each $i=1, \ldots, q$ there exists a positive integer $l_i$ with $\deg f_i +l_i =n_1$ such that 
\begin{equation}\label{containment}
f_i\mathcal{B}_{l_i} \subset [\mathcal{I}_{\mathcal{A}}]_{n_1} \ \ \text{and} \ \ b_j[\mathcal{I}_\mathcal{A}]_{n_{1}-1} \subset [\mathcal{I}_{A}]_{n_1}. 
\end{equation}
for each $j=1, \ldots, p$. Select $n_2$ such that $h \in \mathrm{z.div}(\mathcal{B}_{n_2}/\mathcal{A}_{n_2})$ and $n_2 \gg n_1$. Then there exists $b_{n_2} \in \mathcal{B}_{n_2}$ with $b_{n_2}\not \in \mathcal{A}_{n_2}$ and $hb_{n_2}=0$. Because $b_{n_2} \in \mathcal{I}_{\mathcal{B}}$ we can write $b_{n_2}$ as a combination of the generators $f_i$ of $\mathcal{I}_{\mathcal{B}}$ with coefficients from $\mathcal{B}$. Then (\ref{containment}) implies that 
$b_{n_2} \in \mathcal{A}_{n_2}$ which is a contradiction. Thus $x_0 \in \mathrm{Ass}_{X}(\mathcal{B}_n/\mathcal{A}_n)$ for $n$ large enough. Finally, we get (\ref{stable}) by Prp.\ \ref{finite ass pnts}. 
\end{proof}

Applying Prp.\ \ref{asymptotics algebras} to $\mathcal{A}$ being the image of $\mathcal{R}(\mathcal{M})$ in $\mathcal{B}:=\mathcal{R}(\mathcal{N})$ we get the following corollary.
\begin{corollary}\label{asymptotics}
In the setup of Cor.\ \ref{finitely many} there exists $n_0$ such that
$$\mathrm{Ass}_X(\mathcal{N}^{n_{0}}/\mathcal{M}^{n_{0}})= \mathrm{Ass}_X(\mathcal{N}^{n_{0}+k}/\mathcal{M}^{n_{0}+k})$$ for any  $k$.
\end{corollary}

The next result solves the problem about the asymptotic stability of $\mathrm{Ass}_X(\mathcal{N}^{n}/\overline{\mathcal{M}^{n}})$ in great generality. As a special case, we prove that the sequence of sets $\mathrm{Ass}_X(\mathcal{N}^{n}/\overline{\mathcal{M}^{n}})$  is monotonic with respect to inclusion for each $n$. This recovers classical result due to Ratliff (see Thm.\ 2.4 from \cite{Ratliff84} and Prp.\ 6.8.8 in \cite{Huneke} for another proof that uses valuations) in the case when $\mathcal{M}$ is an ideal in $R$ and by Katz and Naude \cite{Katz1} when $\mathcal{N}$ is a free $R$-module. Our proof is based on the Determinantal Trick idea that we used in the proof of Thm.\ \ref{conormal}.

Let $\mathcal{A} \subset \mathcal{B}$ be finitely generated standard graded $R$-algebras. For each positive integer $n$ denote by $\overline{\mathcal{A}_n}$ the integral closure of $\mathcal{A}_n$ in $\mathcal{B}_n$. Set $\mathcal{A}^{\dagger}:= \oplus_{i=0}^{\infty} \overline{\mathcal{A}_i}$ with $\overline{\mathcal{A}_0}:=R$.
\begin{theorem}\label{int. asym.}
Assume that each minimal prime of $\mathcal{B}$ contracts to a minimal prime of $R$. Then for each $n$
$$\mathrm{Ass}_{X}(\mathcal{B}_n/\overline{\mathcal{A}_n}) \subseteq \mathrm{Ass}_{X}(\mathcal{B}_{n+1}/\overline{\mathcal{A}_{n+1}}).$$
Moreover, for $n$  large enough 
$$\mathrm{Ass}_{X}(\mathcal{B}_n/\overline{\mathcal{A}_n}) = \mathrm{Ass}_{X}(\mathcal{B}_{n+1}/\overline{\mathcal{A}_{n+1}}).$$
\end{theorem}
\begin{proof} 
First, we prove that $\bigcup_{i=1}^{\infty}\mathrm{Ass}_{X}(\mathcal{B}_n/\overline{\mathcal{A}_n})$ is finite following verbatim the argument in the first paragraph of the proof of Thm.\ 2.1 in \cite{Katz1}. Let $I$ be the ideal in $\mathcal{B}$ generated
by $\mathcal{A}_1$. By degree considerations for each $b \in \mathcal{B}_n$ we have: $b$ is in $\overline{\mathcal{A}_n}$ if and
only if $b$ is in $\overline{I^n}$. Then to prove that $\bigcup_{i=1}^{\infty}\mathrm{Ass}_{X}(\mathcal{B}_n/\overline{\mathcal{A}_n})$ is finite one needs to verify that
$\mathrm{Ass}_{\mathcal{B}}(\mathcal{B}/\overline{I^n})$ is finite which follows from the works of Rees \cite{Rees81} and Ratliff \cite{Ratliff84}. 

Next, assume that for some positive integer $n$ we have $x_0 \in \mathrm{Ass}_{X}(\mathcal{B}_n/\overline{\mathcal{A}_n})$ but $x_0 \not \in \mathrm{Ass}_{X}(\mathcal{B}_{n+1}/\overline{\mathcal{A}_{n+1}})$. Replace $X$ by $\mathrm{Spec}(\mathcal{O}_{X,x_0})$ and for each $l$ replace $\mathcal{B}_l$ and $\mathcal{A}_l$ by their stalks at $x_0$. Denote the ideal of $x_0$ by $\mathfrak{m}_{x_0}$. 

Assume $\mathfrak{m}_{x_0}$ is contained in a minimal prime of $\mathcal{B}$. By hypothesis each minimal prime of $\mathcal{B}$ contracts to a minimal prime of $\mathcal{O}_{X,x_0}$. Therefore, $\mathcal{O}_{X,x_0}$ is Artinian local ring. By assumption $x_0 \not \in \mathrm{Ass}_{X}(\mathcal{B}_{n+1}/\overline{\mathcal{A}_{n+1}})$. Because $x_0$ is the only point of $X$, then $\mathcal{B}_{n+1}=\overline{\mathcal{A}_{n+1}}$ which implies that $\mathcal{B}_1$ is integral over $\mathcal{A}$, so $\mathcal{B}=\mathcal{A}^{\dagger}$. In particular, $\mathcal{B}_n=\overline{\mathcal{A}_n}$ contradicting our assumption that $x_0 \in \mathrm{Ass}_{X}(\mathcal{B}_n/\overline{\mathcal{A}_n})$.

From now on, suppose that no minimal prime of $\mathcal{B}$ contains $\mathfrak{m}_{x_0}$. Let $b \in \mathcal{B}_n$ such that $\mathfrak{m}_{x_0}b \in \overline{\mathcal{A}_n}$ and $b \not \in \overline{\mathcal{A}_n}$. Observe that $b\mathcal{A}_{1} \in \mathcal{B}_{n+1}$ and $\mathfrak{m}_{x_0}b\mathcal{A}_{1} \in \overline{\mathcal{A}_{n+1}}$. But $x_0 \not \in \mathrm{Ass}_{X}(\mathcal{B}_{n+1}/\overline{\mathcal{A}_{n+1}})$ by assumption. Hence $b\mathcal{A}_{1} \in \overline{\mathcal{A}_{n+1}}$. Because $\mathcal{A}$ is generated as an $\mathcal{O}_{X,x_0}$-algebra by $\mathcal{A}_1$ we get $b\mathcal{A} \in \mathcal{A}^{\dagger}$. 

Let $c \in \mathcal{A}^{\dagger}$. We claim that $bc \in \mathcal{A}^{\dagger}$. Indeed, $c$ satisfies an equation of integral dependence over $\mathcal{A}:$
$$c^{k}+\alpha_1c^{k-1}+\cdots +\alpha_k=0$$
where $\alpha_i \in \mathcal{A}$. Multiply both sides of the last equation by $b^{k}$ to get
$$(bc)^{k}+b\alpha_1(bc)^{k-1}+\cdots+b^{k}\alpha_k=0.$$
Without loss of generality assume that $c$ is homogeneous. Note that $b\mathcal{A}_1 \in \mathcal{A}^{\dagger}$ and each $\alpha_i$ belongs to $\mathcal{A}_{l(i)}$ for some $l(i) \geq i$. Thus $b^{i}\alpha_i \in \mathcal{A}^{\dagger}$. Because $\mathcal{A}^{\dagger}$ is integrally closed in $\mathcal{B}$ we get $bc \in \mathcal{A}^{\dagger}$. Set $J:= \oplus_{i=1}^{\infty}\overline{\mathcal{A}_i}$. Then
\begin{equation}\label{det. trick}
bJ \subseteq J.
\end{equation}
Denote by $\widehat{\mathcal{O}_{X,x_0}}$ the completion of $\mathcal{O}_{X,x_0}$ with respect to $x_0$. Set 
$$\mathring{\mathcal{A}}=\mathcal{A} \otimes_{\mathcal{O}_{X,x_0}} \widehat{\mathcal{O}_{X,x_0}} \ \text{and} \ \mathring{\mathcal{B}}=\mathcal{B} \otimes_{\mathcal{O}_{X,x_0}} \widehat{\mathcal{O}_{X,x_0}}.$$
Denote by $\mathring{\mathcal{A}}_{\mathrm{red}}$ and $\mathring{\mathcal{B}}_{\mathrm{red}}$  the reductions of $\mathring{\mathcal{A}}$ and $\mathring{\mathcal{B}}$, respectively. Denote by $\mathcal{C}$ the integral closure of $\mathring{\mathcal{A}}_{\mathrm{red}}$ in $\mathring{\mathcal{B}}_{\mathrm{red}}$. Finally, denote by $\mathring{J}$ the image of $J$ in $\mathring{\mathcal{B}}_{\mathrm{red}}$ and identify $b$ with its image in $\mathring{\mathcal{B}}_{\mathrm{red}}$. Because $\mathring{\mathcal{A}}_{\mathrm{red}}$ is finitely generated algebra over complete local ring, then by \cite[\href{http://stacks.math.columbia.edu/tag/0335}{Tag 0335}]{Stacks} it follows that it is Nagata ring (see also Ex.\ 9.7 in \cite{Huneke}). Because $\mathring{\mathcal{B}}_{\mathrm{red}}$ is reduced and finitely generated algebra over $\mathring{\mathcal{A}}_{\mathrm{red}}$, then by \cite[\href{http://stacks.math.columbia.edu/tag/03GH}{Tag 03GH}]{Stacks} or simply by the definition of a Nagata ring (cf.\ Ex.\ 9.6 in \cite{Huneke}), it follows that $\mathcal{C}$ is finitely generated. But $\mathring{J}$ is an ideal in $\mathcal{C}$, so $\mathring{J}$ is a finitely generated $\mathcal{C}$-module.

By (\ref{det. trick}) we have $b\mathring{J} \subseteq \mathring{J}$. Therefore, by the Determinantal Trick Lemma (see Lem.\ 2.1.8 from \cite{Huneke}) there exists a monic polynomial $g(b)$ in $b$ with coefficients in $\mathcal{C}$ such that 
\begin{equation}\label{integral containment}
g(b)\mathring{J}=0. 
\end{equation}

Suppose $\mathring{J}=0$. Then $J = \mathrm{nil}(\mathcal{B}) \cap \oplus_{i>0}\mathcal{B}_{i}$. Because $\mathfrak{m}_{x_0}b \in \overline{\mathcal{A}_n}$, then for each $h \in \mathfrak{m}_{x_0}$ there exists $u$ such that $(hb)^{u}=0$. But $b$ is not a nilpotent, so there exists a minimal prime of $\mathcal{B}$ that does not contain it. Thus $\mathfrak{m}_{x_0}$ is contained in a minimal prime of $\mathcal{B}$ which contradicts our assumptions.

Suppose $\mathring{J}$ is nonzero. If $\mathring{J}$ contains a nonzero divisor in  $\mathring{\mathcal{B}}_{\mathrm{red}}$, then (\ref{integral containment}) implies that $g(b)=0$. Hence $b$ is integral over $\mathcal{C}$. But $\mathcal{C}$ is integrally closed. So $b \in \mathcal{C}$. Suppose $\mathring{J}$ consists of zero divisors in $\mathring{\mathcal{B}}_{\mathrm{red}}$. Then $\mathring{J}$ is contained in some of the minimal primes of $\mathring{\mathcal{B}}_{\mathrm{red}}$ but not in others. Because of (\ref{integral containment}) $g(b)$ is contained in all minimal primes that do not contain $\mathring{J}$. Select $h \in \mathfrak{m}_{x_0}$ that avoids the minimal primes of $\mathring{\mathcal{B}}_{\mathrm{red}}$ and such that $hg(b) \in \mathring{J}$. Then $g(b)$ is contained in the minimal primes of $\mathring{\mathcal{B}}_{\mathrm{red}}$ that contain $\mathring{J}$. Thus, $g(b)$ is contained in the intersection of all minimal primes of  $\mathring{\mathcal{B}}_{\mathrm{red}}$. Therefore, $g(b)=0$ and once again $b \in \mathcal{C}$. 

By Lem.\ \ref{reduction} \rm{(1)} it follows that $b$ is in the integral closure of $\mathring{\mathcal{A}}$ in $\mathring{\mathcal{B}}$. More precisely, $b \in \overline{\mathring{\mathcal{A}_n}}$. As in the proof of Lem.\ \ref{key lemma} $\rm{(1)}$ by faithful flatness we have $\overline{\mathring{\mathcal{A}_n}} \cap \mathcal{B}_n=\overline{\mathcal{A}_n}.$ Hence $b \in \overline{\mathcal{A}_n}$ which is a contradiction.

Finally, we get $x_0 \in \mathrm{Ass}_{X}(\mathcal{B}_{n+1}/\overline{\mathcal{A}_{n+1}})$. Combining the fact that $\bigcup_{i=1}^{\infty}\mathrm{Ass}_{X}(\mathcal{B}_n/\overline{\mathcal{A}_n})$ is finite with the monotonicity property we just proved we get $\mathrm{Ass}_{X}(\mathcal{B}_n/\overline{\mathcal{A}_n}) = \mathrm{Ass}_{X}(\mathcal{B}_{n+1}/\overline{\mathcal{A}_{n+1}})$  for $n$ large enough. The proof of the theorem is now complete.
\end{proof}
\begin{remark}
\rm To prove the monotonicity part of the theorem one might proceed with reducing to the ideal case as Katz and Naude did in their Thm.\ 2.1 for the case when $\mathcal{B}$ is
the symmetric algebra of a free module and $\mathcal{A}$ is the Rees algebra of a module. In our approach we immediately obtain the key inclusion (\ref{det. trick}) which implies that either $b$ satisfies an equation of integral dependence or $J$ is not faithful. In either case we get easily a contradiction with our assumptions. The only technicality is to pass to the case when $J$ is finitely generated. That requires passing through the completion and using faithful flatness as we already did in Lem.\ \ref{key lemma}. 

Another approach to the ideal case that uses valuations can be found in Prp.\ 6.8.8 in \cite{Huneke}.
\end{remark}
\begin{corollary}\label{ass. int. asym.} Assume $\mathcal{M} \subset \mathcal{N}$ are finitely generated $R$-modules. Then for each $n$
$$\mathrm{Ass}_X(\mathcal{N}^{n}/\overline{\mathcal{M}^{n}}) \subseteq \mathrm{Ass}_X(\mathcal{N}^{n+1}/\overline{\mathcal{M}^{n+1}}).$$
Furthermore, there exists $n_0$ such that
$$\mathrm{Ass}_X(\mathcal{N}^{n_{0}}/\overline{\mathcal{M}^{n_{0}}})= \mathrm{Ass}_X(\mathcal{N}^{n_{0}+k}/\overline{\mathcal{M}^{n_{0}+k}})$$
for any $k$.
\end{corollary}
\begin{proof}
Set $\mathcal{B}:=\mathcal{R}(\mathcal{N})$ and let $\mathcal{A}$ be the image of $\mathcal{R}(\mathcal{M})$ in $\mathcal{B}$. Prp.\ \ref{structure morphism} \rm{(1)} applied for $\mathrm{Proj}(\mathcal{B})$ implies that $c_{\mathcal{B}}$ maps each irreducible component of $\mathrm{Proj}(\mathcal{B})$ to an irreducible component of $X$. Then the statement of the corollary follows directly from Thm.\ \ref{int. asym.}.
\end{proof}

Denote by $X_1, \ldots, X_r$ the irreducible components of $X$, and for each $i$ denote by $\eta_i$ the generic point of $X_i$. Let $\mathcal{M} \subset \mathcal{N}$ be a pair of $R$-modules. Assume either 
\begin{enumerate}
\item[(1)] $\mathcal{M}_{\eta_i}$ and $\mathcal{N}_{\eta_i}$ are free $\mathcal{O}_{X,\eta_i}$-modules of ranks $e_i$ and $p_i$,
or
\item[(2)] $\mathcal{M}$ and $\mathcal{N}$ are contained in a free $R$-module $\mathcal{F}$.
\end{enumerate}
As usual set $B:= \mathrm{Proj}(\mathcal{R}(\mathcal{M}))$ and denote by $b$ the structure morphism $b \colon B \rightarrow X$. Let $B_i$ be the irreducible component of $B$ that surjects onto $X_i$. Denote by $b_i$ the restriction of $b$ to $B_i$. Let $S_{i}(k)$ be the closed subschemes of $X$ consisting of all points $x \in X_{i}$ such that $b_{i}^{-1}x \geq k$. Finally, define $Z$ to be the nonfree locus of $\mathcal{N}/\mathcal{M}$. 

\begin{theorem}\label{ass points-int.cl.} The following hold:
\begin{itemize}
\item[(1)] If $p_i>e_i$, then $\eta_i \in \mathrm{Ass}_{X}(\mathcal{N}^{n}/\overline{\mathcal{M}^{n}})$ for each $n$. If $p_i=e_i$, then $\eta_i \not \in \mathrm{Ass}_{X}(\mathcal{N}^{n}/\overline{\mathcal{M}^{n}})$ for each $n$.
\item[(2)] Suppose $x \neq \eta_i$ for each $i$. Assume  
$x \in \mathrm{Ass}_{X}(\mathcal{N}^{n}/\overline{\mathcal{M}^{n}})$. Then $x \in Z$. Additionally, if $\codim(\overline{\{x\}},X_i)=1$ for some $X_i$, then $x$  is the generic point of
an irreducible component of $Z$.
\end{itemize}
If $\codim(\overline{\{x\}},X_i)>1$ for some $i$ and $x \in \mathrm{Ass}_{X}(\mathcal{N}^{n}/\overline{\mathcal{M}^{n}})$ for some $n$, and $X$ is universally catenary, then $x$ is the generic point of an irreducible component of $S_{i}(\dim X_{i},x+e_{i}-2)$. In particular, the set $\bigcup_{n=1}^{\infty}\mathrm{Ass}_{X}(\mathcal{N}^{n}/\overline{\mathcal{M}^{n}})$ is finite.
\end{theorem}

\begin{proof}
By Lem.\ \ref{reduction} we can pass to the reductions of $X$, $\mathcal{R}(\mathcal{M})$ and $\mathcal{R}(\mathcal{N})$.

Consider $\rm{(1)}$. If $p_i>e_i$, then $(\mathcal{N}^{n}/\overline{\mathcal{M}^{n}})_{\eta_i} \simeq (\mathcal{O}_{X,\eta_i})^{l_i(n)}$ where $l_{i}(n): = p_{i}(n)-e_{i}(n)$ is positive. 
Indeed, the integral closure of $(\mathcal{O}_{X,\eta_i})^{e_i(n)}$ in 
$(\mathcal{O}_{X,\eta_i})^{p_i(n)}$ is  $(\mathcal{O}_{X,\eta_i})^{e_i(n)}$ because $X$ is reduced and the former module is a direct summand of the latter as observed in Prp.\ \ref{support}. 
If $p_i=e_i$, then $\mathcal{N}_{\eta_{i}}^{n}=\overline{\mathcal{M}_{\eta_i}^{n}}$ for each $n$. This proves $\rm{(1)}$. 

Consider $\rm{(2)}$. Suppose that $x \not \in Z$. Then the module $\mathcal{N}_{x}^{n}/\overline{\mathcal{M}_{x}^{n}}$ equals to the quotient of $(\mathcal{O}_{X,x})^{l_i(n)}/\mathcal{L}_n$ for some $\mathcal{L}_n$. However, locally at each $\eta_j$ such that $\overline{\{x\}} \subset \overline{\{\eta_j\}}$ we have $\mathcal{N}_{\eta_j}^{n}/\overline{\mathcal{M}_{\eta_j}^{n}}\simeq (\mathcal{O}_{X,\eta_j})^{l_i(n)}$. Therefore, because $X,x$ is reduced and $\mathcal{L}_n$ is contained in a free $\mathcal{O}_{X,x}$-module, then $\mathcal{L}_n=0$. So $\mathcal{N}_{x}^{n}/\overline{\mathcal{M}_{x}^{n}}=(\mathcal{O}_{X,x})^{l_i(n)}$. Because $x \neq \eta_i$ for each $i$, then $x$ must be an embedded point of $X$. But this is impossible because $X$ is reduced. Therefore, $x \in Z$. In the proof of Prp.\ \ref{support} we showed that $Z$ does not contain $\eta_i$. Thus  $x$ is the generic point of an irreducible component of $Z$.

Next, suppose that  $\codim(\overline{\{x\}},X_i)>1$ for some $i$ and $x \in \mathrm{Ass}_{X}(\mathcal{N}^{n}/\overline{\mathcal{M}^{n}})$ for some $n$. By Thm.\ \ref{conormal} and Prp.\ \ref{structure morphism} \rm{(2)} and \rm{(3)} there exists $i$ such that $\dim b_{i}^{-1}x = \dim X_{i},x+e_i-2$. Suppose there exists $y \in X_i$ such that $\overline{\{x\}} \subset \overline{\{y\}}$ and $y \in S_{i}(\dim X_{i},x+e_i-2)$. Note that $y$ can not be the generic point of $X_i$ for otherwise $y \in S_i(e_i-1)$ and hence $\dim X_i,x=1$ which contradicts with our assumption. Thus $X_i,y$ is of positive dimension. By Prp.\ \ref{structure morphism} \rm{(3)} $$\dim b_{i}^{-1}y \leq \dim X_i,y+e_i-2.$$ Therefore, we must have $\dim X_i,y \geq \dim X_i,x$. But this is impossible because $\overline{\{x\}} \subset \overline{\{y\}}$ and hence $\dim X_{i},x > \dim X_{i},y$. This proves that $x$ is the generic point of an irreducible component of $S_{i}(\dim X_{i},x+ e_{i}-2)$ for some $i$. The sets $S_{i}(k)$ are  closed in $X$ by Chevalley's theorem and finitely many because $B$ is of finite type over a Noetherian scheme; hence $B$ is Noetherian. Therefore, $\bigcup_{n=1}^{\infty}\mathrm{Ass}_{X}(\mathcal{N }^{n}/\overline{\mathcal{M}^{n}})$ is a finite set.
\end{proof}

\section{A converse to the main result}\label{converse results}
Adopt the setup from the beginning of Sct.\ \ref{main results}. Our principal goal is to prove converses of Thm.\ \ref{conormal} and Thm.\ \ref{ass points-int.cl.} without requiring $X$ to be universally catenary. However, a quick inspection reveals that the module $\mathcal{N}$ has to be of particular type for such converses to hold. 
\begin{definition}
We say that $\mathcal{N}$ has {\it deficient analytic spread} at $x \in X$ if  each irreducible component of the fiber of $\mathrm{Proj}(\mathcal{R}(\mathcal{N}_x))$ over $x$ is of codimension at least $2$ in $\mathrm{Proj}(\mathcal{R}(\mathcal{N}_x))$.
\end{definition}
Note that if $\dim X,x \geq 2$, then every free $R$-module has deficient analytic spread at $x$. When $X,x_0$ is universally catenary, irreducible and local with closed point $x_0$, $\mathcal{N}$ has deficient analytic spread (at $x_0$) if and only if $\mathcal{N}$ does not have the maximal possible analytic spread (see Cor.\ 16.4.7 in \cite{Huneke} for definition of analytic spread and Rmk.\ \ref{def. vs. anal.} below). Deficient analytic spread implies nonmaximal analytic spread, but in general without additional hypothesis on $X$ the converse may not be true. 
Combining the results of this section with those from previous sections we will be able to classify the sets $\mathrm{Ass}_{X}(\mathcal{N}^{n}/\overline{\mathcal{M}^n})$ for modules $\mathcal{N}$ which have deficient analytic spread. 


We begin with Thm.\ \ref{converse conormal} which is the main result of this section. It is a fairly general statement about finitely generated graded algebras. As a first application we prove the converse to Thm.\ \ref{conormal} assuming that $\mathcal{N}$ has deficient analytic spread. As in Thm.\ \ref{conormal}  $\dim X \geq 2$ is assumed implicitly.  In the case when $X,x_0$ is of dimension one and $\mathcal{N}:=\mathcal{F}$ is free, we prove in Prp.\ \ref{curve} that if $x_0 \notin \mathrm{Ass}_{X}(\mathcal{F}/\overline{\mathcal{M}})$, then $\mathcal{M}$ and $\mathcal{F}$ share a free direct summand of maximal rank. It's a result that we will use on several occasions later in the section. As a first application of  Thm.\ \ref{converse conormal} we prove in Cor.\ \ref{converse to KT} a partial converse to the Kleiman--Thorup theorem, which can be viewed as a criterion for nonintegrality of modules.

Another important application of the main result of this section is Thm.\ \ref{compl. desc.}: we prove a converse to Thm.\ \ref{ass points-int.cl.} assuming that $\mathcal{N}$ has deficient analytic spread at each point $x \in X$ with $\dim X,x \geq 2$. Our result generalizes results by Katz and Rice (see Thm.\ 3.5.1 in \cite{Katz2}) who assumed that $\mathcal{N}:=\mathcal{F}$ is free and $\mathrm{rk}(\mathcal{F})=e_i$ for each $i=1, \ldots, r$, where $e_i$ is the rank of $\mathcal{M}$ at the generic point of $X_i$, and by  Rees \cite{Rees} who assumed additionally that $R$ is locally quasi-unmixed. Furthermore, we will recover a result by Burch \cite{Burch} in the ideal case (cf.\ Thm.\ 5.4.7 in \cite{Huneke}) which is not recovered by the results of Rees, Katz and Rice. We conclude this section with an example arising  from geometry for which $\mathcal{M}$ has deficient analytic spread and $\mathrm{Ass}_{X}(\mathcal{F}^{n}/\overline{\mathcal{M}^{n}})$ consists entirely of the generic points of the irreducible components of $X$. This class of examples arises from projective varieties with deficient duals and motivates our coining of the term ``deficient analytic spread." 


Let $X$ be a local Noetherian scheme with closed point $x_0$ and let $\mathcal{A} \subset \mathcal{B}$ be finitely generated $\mathbb{N}_0$-graded $\mathcal{O}_{X,x_0}$-algebras. Denote their graded pieces by $\mathcal{A}_n$ and $\mathcal{B}_n$, respectively. Assume $\mathcal{A}_0=\mathcal{B}_0=\mathcal{O}_{X,x_0}$. Denote by $\overline{\mathcal{A}_n}$ the integral closure of $\mathcal{A}_n$ in $\mathcal{B}_n$. Set $\mathcal{A}^{\dagger}:=\bigoplus_{i=0}^{\infty}\overline{\mathcal{A}_i}$ with $\overline{\mathcal{A}_0}:=\mathcal{O}_{X,x_0}$. Denote by $\mathfrak{m}_{x_0}$ the maximal ideal of $\mathcal{O}_{X,x_0}$.
\begin{definition}
Say that $\mathcal{B}$ has {\it deficient analytic spread} if the codimension (height) of each minimal prime ideal of $\mathfrak{m}_{x_0}\mathcal{B}$ is at 
least $2$. 
\end{definition}
Note that because $\mathfrak{m}_{x_0}\mathcal{B}$ is homogeneous, each of its minimal primes has the same codimension in $\mathrm{Spec}(\mathcal{B})$ as in $\mathrm{Proj}(\mathcal{B})$. \begin{theorem}\label{converse conormal}
Suppose that $\bigcup_{n=1}^{\infty}\mathrm{Ass}_{X}(\mathcal{B}_n/\overline{\mathcal{A}_n})$ is finite and $x_0 \notin \mathrm{Ass}_{X}(\mathcal{B}_n/\overline{\mathcal{A}_n})$ for $n$ large enough. If $\mathcal{B}$ has deficient analytic spread, then $\mathcal{A}$ has deficient analytic spread.  
\end{theorem}
\begin{proof} 
If $\dim X \leq 1$, then by Krull's height theorem each minimal prime of $\mathfrak{m}_{x_0}\mathcal{B}$ is of codimension at most $1$. But $\mathcal{B}$ has deficient analytic spread. Then $\mathcal{O}_{X,x_0}$ must have a minimal prime of codimension at least $2$. 

By Lem.\ \ref{reduction} $\rm{(3)}$ and the fact that passing to reduced structures does not affect our codimension hypothesis, we can assume that $X$, $\mathcal{A}$ and $\mathcal{B}$ are reduced. By assumption no minimal prime of $\mathcal{B}$ contracts to  the maximal ideal $\mathfrak{m}_{x_0}$ of $\mathcal{O}_{X,x_0}$. Hence no minimal prime of $\mathcal{A}$ contracts to $\mathfrak{m}_{x_0}$ either. 

Because $x_0 \notin \mathrm{Ass}_{X}(\mathcal{B}_n/\overline{\mathcal{A}_n})$ for $n$ large enough, $\bigcup_{n=1}^{\infty}\mathrm{Ass}_{X}(\mathcal{B}_n/\overline{\mathcal{A}_n})$ is finite by assumption,
by prime avoidance we can select $h_1$ from $\mathfrak{m}_{x_0}$
such that $h_1 \notin \mathrm{z.div}(\mathcal{B}_n/\overline{\mathcal{A}_n})$ for $n$ large enough and $h_1$ avoids the minimal primes of $\mathcal{B}$ and $\mathcal{A}$. By Krull's height theorem each minimal prime of $h_{1}\mathcal{A}$ in $\mathcal{A}$ is of codimension one.

Consider the map $$\psi_{h_1} \colon \mathcal{A}/h_{1}\mathcal{A} \rightarrow \mathcal{B}/h_{1}\mathcal{B}.$$ Set $\mathcal{A}(h_1):= \Ima \psi_{h_{1}}$. Observe that $\Ker \psi_{h_1} = (h_{1}\mathcal{B} \cap \mathcal{A})/h_{1}\mathcal{A}$. By hypothesis $h_{1} \notin \mathrm{z.div}(\mathcal{B}_{n}/\overline{\mathcal{A}_{n}})$ for $n$ large enough. If $b \in \mathcal{B}_k$ with $h_{1}b \in \mathcal{A}_k$ for some $k$, then for each $l$ we have $h_{1}^{l}b^{l} \in \mathcal{A}_{kl}$. But for $l$ large enough we have $h_{1} \notin \mathrm{z.div}(\mathcal{B}_{lk}/\overline{\mathcal{A}_{lk}})$, so $b^{l} \in \overline{\mathcal{A}_{kl}}$.  
Then
$$(b^{l})^{s}+a_1(b^{l})^{s-1}+\cdots+a_s=0$$
for $a_i \in \mathcal{A}_{kli}$. Multiplying both sides of the last equation by $(h_{1})^{ls}$ we get
$$(h_{1}b)^{ls}+h_{1}^{l}a_{1}(h_{1}b)^{l(s-1)}+ \cdots + h_{1}^{ls}a_s=0.$$
Therefore, $(h_{1}b)^{ls}=0$ in $\mathcal{A}/h_{1}\mathcal{A}$ because $h_{1}^{li}a_{i}(h_{1}b)^{l(s-i)} \in h_{1}\mathcal{A}$ for each $i=1, \ldots, s$. Thus, $\Ker \psi_{h_1}$ consists of nilpotents, so $\mathcal{A}/h_{1}\mathcal{A}$ and  $\mathcal{A}(h_1)$ have the same reduced structures.

We have $\mathcal{A}(h_1)_{\mathrm{red}} \subset (\mathcal{B}/h_{1}\mathcal{B})_{\mathrm{red}}$. By hypothesis, no minimal prime of  $h_{1}\mathcal{B}$ contracts to $\mathfrak{m}_{x_0}$. By prime avoidance there exists  $h_2 \in \mathfrak{m}_{x_0}$ such that $h_2 \not \in \mathrm{z.div}(\mathcal{B}/h_{1}\mathcal{B})_{\mathrm{red}}$ and hence 
$h_2 \not \in \mathrm{z.div}\mathcal{A}(h_1)_{\mathrm{red}}$. So $h_2$ avoids the minimal primes of $\mathcal{A}/h_{1}\mathcal{A}$. Therefore, by Krull's height theorem the codimension of each minimal prime of the ideal $\langle h_1,h_2 \rangle \mathcal{A}$ in $\mathcal{A}$ is $2$. But $\langle h_1,h_2 \rangle \mathcal{A} \subset \mathfrak{m}_{x_0}\mathcal{A}$. Thus $\mathcal{A}$ has deficient analytic spread. 
\end{proof}

Note that if $\dim X \ge 2$ and $\mathcal{B}$ is a standard polynomial ring, then the hypothesis on the codimension of the special fiber of $\mathcal{B}$ is satisfied. Furthermore, in this case if $\dim X_i \ge 2$ for each $i$, then the proof above shows that the codimension of $\mathfrak{m}_{x_0}\mathcal{A}$ is at least $2$. 

We will show elsewhere that every standard graded algebra $\mathcal{A}$, whose minimal primes contract to minimal primes of $X$ with $\dim X \geq 2$, can be embedded into a standard graded algebra $\mathcal{B}$ generically equal to $\mathcal{A}$ which has deficient analytic spread. 

Assume $X$ is local Noetherian scheme with closed point $x_0$. Suppose that $\mathcal{M} \subset \mathcal{N}$ is a pair of $\mathcal{O}_{X,x_0}$-modules satisfying the hypothesis at the beginning of Sct.\ \ref{main results}. Applying Thm.\ \ref{converse conormal} and Thm.\ 2.1 in \cite{Katz1} with $\mathcal{B}:=\mathcal{R}(\mathcal{N})$ and $\mathcal{A}$ equal to the image of $\mathcal{R}(\mathcal{M})$ in $\mathcal{B}$  we get the following general converse of Thm.\ \ref{conormal}.
\begin{corollary}\label{converse deficient}
Suppose $\mathcal{N}$ has deficient analytic spread and $x_0 \not \in \mathrm{Ass}_{X}(\mathcal{N}^{n}/\overline{\mathcal{M}^{n}})$ for each $n$. Then $\mathcal{M}$ has deficient analytic spread, too. 
\end{corollary}
The deficient analytic spread hypothesis implicitly assumes that $\dim X,x_0 \geq 2$. If $\mathcal{N}:=\mathcal{F}$ is a free module, then we can say what happens in the one-dimensional case. The next result is a criterion for a point $x_0$ from $X$ to be an associated point of $\mathcal{F}/\overline{\mathcal{M}}$ when $\dim X,x_0 =1$. In this setting we will assume that $\mathcal{M}$ has rank $e_i$ at the generic point of $X_i$. This is a natural assumption because by Lem.\ \ref{key lemma} we can study the set $\mathrm{Ass}_{X}(\mathcal{F}^{n}/\overline{\mathcal{M}^{n}})$ by passing to the reductions of $\mathrm{Sym}(\mathcal{F})$ and $\mathcal{R}(\mathcal{M})$ and noting that the image of
$\mathcal{M}$ in $\mathcal{F} \otimes_{\mathcal{O}_{X}} \mathcal{O}_{{X}_{\mathrm{red}}}$ is free at the generic point of each $(X_i)_\mathrm{red}$
Set  $e:=\max_{i=1}^{r}\{e_i\}$ and $p:=\mathrm{rk}(\mathcal{F})$. 
\begin{proposition}\label{curve}
Assume $X$ is local scheme with closed point $x_0$ and $\dim X =1$. Assume that $x_0 \notin \mathrm{Ass}_{X}(\mathcal{F}/\overline{\mathcal{M}})$. Then $e_i=e$ for each $i=1, \ldots, r$ and there exists a free module $G$ of rank $e$ that is a direct summand of $\mathcal{M}$ and $\mathcal{F}$. 

If $X$ is reduced or $p=e$, then $\mathcal{M}$ is free of rank $e$ and a direct summand of $\mathcal{F}$. 
\end{proposition}
\begin{proof} Let $t$ be an element from the maximal ideal of the local ring $\mathcal{O}_{X,x_0}$ that avoids
its minimal primes. The inclusion $\mathcal{R}(\mathcal{M}) \hookrightarrow \mathcal{R}(\overline{\mathcal{M}})$ induces
a finite map $\mathrm{Proj}(\mathcal{R}(\overline{\mathcal{M}})) \rightarrow \mathrm{Proj}(\mathcal{R}(\mathcal{M}))$ which gives rise to a finite map $\mathrm{Proj}(\mathcal{R}(\overline{\mathcal{M}})/t\mathcal{R}(\overline{\mathcal{M}})) \rightarrow \mathrm{Proj}(\mathcal{R}(\mathcal{M})/t\mathcal{R}(\mathcal{M}))$. But by Prp.\ \ref{structure morphism} \rm{(3)} and the fact that $X$ is universally catenary because $\dim X=1$  (see Cor. 2 to Thm.\  31.7 from \cite{Matsumura}) we have  $\dim \mathrm{Proj}(\mathcal{R}(\mathcal{M})/t\mathcal{R}(\mathcal{M}))=e-1$. Hence
\begin{equation}\label{dim. int.}
 \dim \mathrm{Proj}(\mathcal{R}(\overline{\mathcal{M}})/t\mathcal{R}(\overline{\mathcal{M}})) =e-1.
\end{equation}
By hypothesis  $x_0 \notin \mathrm{Ass}_{X}(\mathcal{F}/\overline{\mathcal{M}})$. Therefore, $\overline{\mathcal{M}}/t\overline{\mathcal{M}}$ injects 
into $\mathcal{F}/t\mathcal{F}$. Thus $\mathcal{R}(\overline{\mathcal{M}})/t\mathcal{R}(\overline{\mathcal{M}})$ is a subalgebra 
generated in degree one of a polynomial algebra over the Artinian ring $\mathcal{O}_{X,x_0}/t\mathcal{O}_{X,x_0}$.  Consider the reduced structures of $\mathrm{Sym}(\mathcal{F}/t\mathcal{F})$ and $\mathcal{R}(\overline{\mathcal{M}})/t\mathcal{R}(\overline{\mathcal{M}})$.
Because $(\mathcal{R}(\overline{\mathcal{M}})/t\mathcal{R}(\overline{\mathcal{M}}))_{\mathrm{red}}$ is generated in degree one inside a polynomial ring over the residue field $k(x_0)$ and because of (\ref{dim. int.}) it follows that there exists $e$ elements in $\overline{\mathcal{M}}$ whose images in $\mathcal{F}/\mathfrak{m}_{x_0}\mathcal{F}$ are linearly independent over $k(x_0)$. Thus, by Nakayama's lemma there exist a free module $G$ of rank $e$, a free module $\mathcal{F}_{1}$, and $H \subset \mathfrak{m}_{x_0}\mathcal{F}_{1}$ such that 
$$\overline{\mathcal{M}}=G \oplus H \ \text{and} \ \mathcal{F}= G \oplus \mathcal{F}_{1}.$$
Hence $\dim (\mathcal{F}/\overline{\mathcal{M}}) \otimes k(\eta_i) \leq p-e$. However, $\mathcal{F}_{\eta_i}/\overline{\mathcal{M}_{\eta_i}} = k(\eta_i)^{p-e_i}$. Therefore, $e_i=e$ for each $i=1, \ldots, r$. In particular, 
$H_{\eta_i}=0$ for each $i$. But $H$ is contained in a free $\mathcal{O}_{X,x_0}$-module. Thus, $H \subset \mathrm{nil}(\mathcal{O}_{X,x_0})\mathcal{F}_{1}$. On the other hand, $\mathrm{nil}(\mathcal{O}_{X,x_0})\mathcal{F}_{1} \subset \overline{\mathcal{M}}.$ Hence $H=\mathrm{nil}(\mathcal{O}_{X,x_0})\mathcal{F}_{1}$. Because each element of $G$ is integral over $\mathcal{M}$, then $G$ must be a direct summand of $\mathcal{M}$ as well (cf. Lem.\ 2.4.1 in \cite{Katz2}).

 If $p=e$, then $\mathcal{M}=\mathcal{F}$. Suppose $X$ is reduced. Then $H=0$. Hence $\mathcal{M}=G$. Thus $\mathcal{M}$ is free of rank $e$ and direct summand of $\mathcal{F}$. 
\end{proof}

\begin{remark}\label{def. vs. anal.} \rm Before going further we would like to compare the hypothesis and implications of our Thm.\ \ref{conormal} and Cor.\  \ref{converse deficient} with those in the aforementioned results of McAdam, Burch, Rees, Katz and Rice. To facilitate our discussion let's focus on Thm.\ 3.5.1 in \cite{Katz2} which is a special case of Thm.\ \ref{conormal} and Cor.\  \ref{converse deficient}. For simplicity assume $X,x_0$ is local and irreducible, $\mathcal{N}:=\mathcal{F}$ is a free $\mathcal{O}_{X,x_0}$-module and $\mathcal{M}$ has rank equal to $\mathrm{rk}(\mathcal{F})$ at the generic point of $X$. Define the {\it analytic spread} $l(\mathcal{M})$ of $\mathcal{M}$ as the Krull dimension of $\mathcal{R}(\mathcal{M})/\mathfrak{m}_{x_0}\mathcal{R}(\mathcal{M})$. Observe that if $\dim X \leq 1$, then Prp.\ \ref{curve},  Lem.\ 2.4.1 in \cite{Katz2} and Thm.\ \ref{int. asym.} imply that $x_0 \in \mathrm{Ass}_{X}(\mathcal{F}^n/\overline{\mathcal{M}^n})$ for each $n$ provided that $\mathcal{M} \neq \mathcal{F}$. Thus we can assume that $\dim X \geq 2$. 

In Thm.\ \ref{conormal} we assume that $X$ is universally catenary and $\codim D_j \geq 2$ in $\mathcal{R}(\mathcal{M})$. 
The appropriate restatement of the hypothesis of the reverse direction of Thm.\ 3.5.1 in \cite{Katz2} is $X$ is universally catenary and $l(\mathcal{M})$ is not maximal. The two sets of hypothesis are the same as demonstrated by Prp.\ \ref{structure morphism} $\rm{(3)}$. However, in Cor.\ \ref{converse deficient} assuming that $x_0 \not \in \mathrm{Ass}_{X}(\mathcal{F}^n/\overline{\mathcal{M}^n})$ for each $n$, we get that $\mathcal{M}$ has deficient analytic spread or in other words 
\begin{equation}\label{hyp.1}
\codim D_j \geq 2
\end{equation}
for each $j$, whereas the conclusion of the forward direction of Thm.\ 3.5.1 in \cite{Katz2} is
that 
\begin{equation}\label{hyp.2}
l(\mathcal{M}) \leq \dim \mathcal{R}(\mathcal{M})-2,
\end{equation}
or in other words $l(\mathcal{M})$ is not maximal. Because 
\begin{equation}\label{hyp.3}
l(\mathcal{M}) + \codim D_j \leq \dim  \mathcal{R}(\mathcal{M}), 
\end{equation}
with equality when $X$ is universally catenary, we see that (\ref{hyp.1}) implies (\ref{hyp.2}) always. However, (\ref{hyp.2}) may not imply  (\ref{hyp.1})  in the case when (\ref{hyp.3}) is a strict inequality. Therefore, the conclusion of Cor.\ \ref{converse deficient} is stronger than that of the forward direction of Thm.\ 3.5.1 in \cite{Katz2}. The same comparison applies for the results of McAdam and Burch in the ideal case.
\end{remark}


Let $X$ be a local, Noetherian reduced scheme. Consider the nested chain of modules  $\mathcal{M} \subsetneq \mathcal{N} \subset \mathcal{F}$ where $\mathcal{F}$ is free and $\mathcal{M}$ and $\mathcal{N}$ are locally free of the same rank $e_i$ at the generic point of each irreducible component $X_i$ of $X$. For $x \in X$, let  $e_x$ be the maximal $e_i$ where $i$ runs through all indices for which $x \in X_i$. Denote by $Z$ the nonfree locus of  $\mathcal{F}/\mathcal{M}$. Identify the Rees algebras $\mathcal{R}(\mathcal{M})$ and $\mathcal{R}(\mathcal{N})$ with their images in $\mathrm{Sym}(\mathcal{F})$. For each $n$ denote by $\overline{\mathcal{M}^{n}}$ the integral closure of $\mathcal{M}^{n}$ in $\mathcal{F}^{n}$. For each closed set $W$ of $X$ set
$$\mathcal{M}^{n}(W):= \bigcup_{l=0}^{\infty}(\mathcal{M}^{n} :_{\mathcal{F}^{n}} \mathcal{I}_{W}^{l})$$
where $\mathcal{I}_{W}$ is the ideal of $W$ in $\mathcal{O}_X$. Denote by $n_0$ the smallest integer for which $$\mathrm{Ass}_{X}(\mathcal{F}^{n_0}/\overline{\mathcal{M}^{n_0}})=\mathrm{Ass}_{X}(\mathcal{F}^{n_0+k}/\overline{\mathcal{M}^{n_0+k}})$$
for every $k$. Its existence is guaranteed by Cor.\ \ref{ass. int. asym.}. Finally, set $B:= \mathrm{Proj}(\mathcal{R}(\mathcal{M}))$ and denote by $b$ the structure morphism $b \colon B \rightarrow X$. For this setup we can use Thm.\ \ref{converse conormal} to get a converse to Thm.\ \ref{Kleiman-Thorup}.
\begin{corollary}\label{converse to KT}
Let $Z_1$ be an irreducible component of $Z$ with generic point $x_0$. Suppose that $\mathcal{N}^{n}$ contains $\mathcal{M}^{n}(Z_1)$ for some $n>n_0$. We have:
 \begin{itemize}
     \item [(1)] If $\codim (Z_1,X_j)=1$ for all $X_j$ that contain $Z_1$ and $\mathcal{M}$ and $\mathcal{F}$ do not share a free direct summand of rank $e_{x_0}$, then $\mathcal{N}$ is not integral over $\mathcal{M}$. 
     \item[(2)]  If $\codim (Z_1,X_j) \neq 1$ for some $X_j$ that contains $Z_1$ and  $\dim b^{-1}x_0 = \dim X_i,x_0+e_i-2$ for some $X_i$ that contains $x_0$, then  $\mathcal{N}$ is not integral over $\mathcal{M}$.
 \end{itemize}
\end{corollary}
\begin{proof} First note that $\mathrm{Supp}_{X,x_0}(\overline{\mathcal{M}^{n}}/\mathcal{M}^{n})=x_0$. Indeed, if $y \in X,x_0$ with $y \neq x_0$, then $\mathcal{M}_y$ is free and direct summand of $\mathcal{F}_y$. Hence $\mathcal{M}_y$ is integrally closed because $X$ is reduced.

Consider $\rm{(1)}$. By assumption $\codim (Z_1,X_j)=1$ for all $X_j$ that contain $Z_1$. Thus $X,x_0$ is of pure dimension one. Replace $X$ with $X,x_0$ and replace $\mathcal{M}$, $\mathcal{N}$ and  $\mathcal{F}$ with their stalks at $x_0$. By Prp.\ \ref{curve} and our assumptions $x_0 \in \mathrm{Ass}_{X}(\mathcal{F}/\overline{\mathcal{M}})$. By Cor.\ \ref{ass. int. asym.} $x_0 \in \mathrm{Ass}_{X}(\mathcal{F}^n/\overline{\mathcal{M}^n})$ for each $n$. Then $\mathcal{M}^{n}(x_0)$ contains $\overline{\mathcal{M}^{n}}$ for each $n$. Because $\mathcal{N}^{n}$ contains $\mathcal{M}^{n}(x_0)$ for some $n$ by hypothesis, it follows that $\mathcal{N}$ is not integral over $\mathcal{M}$.


Consider $\rm{(2)}$. 
By assumption 
\begin{equation}\label{dim-cat}
\dim b^{-1}x_0 = \dim X_i,x_0+e_i-2.
\end{equation}
Replace $X$ and $X_i$ with $X,x_0$ and  $X_i,x_0$. Further, replace $\mathcal{M}$, $\mathcal{N}$ and  $\mathcal{F}$ with their stalks at $x_0$. We have $\dim X,x_0 \ge 2$ because by assumption $\dim X,x_0 \neq 1$ and $x_0$ can not be a generic point of an irreducible component of $X$. Let $n>n_0$ be such that $\mathcal{N}^{n}$ contains $\mathcal{M}^{n}(x_0)$. Then $\overline{\mathcal{M}^{n}} \subset \mathcal{N}^{n}$. Assume $\overline{\mathcal{M}^{n}} = \mathcal{N}^{n}$. Then $x_0 \not \in \mathrm{Ass}_{X}(\mathcal{F}^{n}/\overline{\mathcal{M}^{n}})$ for each $n$ large enough. Hence by Thm.\ \ref{converse conormal} for each irreducible component $D_j$ of $b_{i}^{-1}x_0$ we have $\codim D_j \geq 2$ in each irreducible component of $B$ that contains $D_j$; hence $\dim b^{-1}x_0 \leq \dim X_{i},x_0+e_i-3$ contradicting (\ref{dim-cat}). Thus $\overline{\mathcal{M}^{n}} \neq \mathcal{N}^{n}$ and hence $\mathcal{N}$ is not integral over $\mathcal{M}$.
\end{proof}
Of course, if one wants to keep the analogy with Thm.\ \ref{Kleiman-Thorup}, then in Cor.\ \ref{converse to KT} $(\rm{2})$ we can replace the dimension hypothesis with $\dim b^{-1}Z_1 = \dim X_i+e_i-2$ provided that $X$ is catenary. 

Next, assume $\mathcal{N}$ has rank $p_i$  and $\mathcal{M}$ has rank $e_i$ at the generic point of each irreducible component $X_i$ of $X$. As above, for $x \in X$, let  $e_x$ be the maximal $e_i$ where $i$ runs through all indices for which $x \in X_i$. Set $B:=\mathrm{Proj}(\mathcal{M})$. For each irreducible component $B_i$ of $B$ define the closed sets $S_{i}(k)$ as in (\ref{S(i)}). The following result generalizes  results by  Rees \cite{Rees}, Katz and Rice (see Thm.\ 3.5.1 in \cite{Katz2}) and Burch \cite{Burch} in the ideal case. Along with  Cor.\ \ref{McAdam} it provides complete discription of the associated points of $\mathcal{N}^{n}/\overline{\mathcal{M}^{n}}$.
\begin{theorem}\label{compl. desc.}
The following hold:
\begin{itemize}
\item[(1)] If $\eta_i$ is the generic point of $X_i$ for some $i$ and $p_i \neq e_i$, then $\eta_i \in \mathrm{Ass}_{X}(\mathcal{N}^{n}/\overline{\mathcal{M}^{n}})$ for each $n$.
\item[(2)] Suppose $\mathcal{N}:=\mathcal{F}$ is a free $R$-module. If $\dim X,x =1$, and $\mathcal{F}$ and $\mathcal{M}$ do not share a direct free summand of rank $e_x$, then $x \in \mathrm{Ass}_{X}(\mathcal{F}^{n}/\overline{\mathcal{M}^{n}})$ for each $n$.
\item[(3)] Assume $\mathcal{N}$ has deficient analytic spread at each $x \in X$ with $\dim X,x \geq 2$. If $\dim X,x \geq 2$ and $x$ is the generic point of an irreducible component of $S_{i}(\dim X_i,x+e_i-2)$ for some $i$,  then $x \in \mathrm{Ass}_{X}(\mathcal{N}^{n}/\overline{\mathcal{M}^{n}})$ for $n$ large enough.
\end{itemize}
\end{theorem}
\begin{proof}
We have $x \in \mathrm{Ass}_{X}(\mathcal{N}^n/\overline{\mathcal{M}^n})$ if and only if $ x \in \mathrm{Ass}_{X,x}(\mathcal{N}_{x}^n/\overline{\mathcal{M}_{x}^n})$. So assume $X$ is local with closed point $x$. Replace $\mathcal{M}$ and $\mathcal{N}$ with their stalks at $x$. Part $\rm{(1)}$ follows from the proof of Thm.\ \ref{ass points-int.cl.} $\rm{(1)}$. 

Consider $\rm{(2)}$. By Prp.\ \ref{curve} we have $x \in \mathrm{Ass}_{X,x}(\mathcal{F}_x/\overline{\mathcal{M}_x})$, so by Cor.\ \ref{ass. int. asym.} $x \in \mathrm{Ass}_{X}(\mathcal{F}^{n}/\overline{\mathcal{M}^{n}})$ for each $n$. 

Consider $\rm{(3)}$. Suppose $x \notin \mathrm{Ass}_{X}(\mathcal{N}^{n}/\overline{\mathcal{M}^{n}})$ for each $n$. By Cor.\ \ref{ass. int. asym.} $\bigcup_{n=1}^{\infty}\mathrm{Ass}_{X}(\mathcal{N}^{n}/\overline{\mathcal{M}^{n}})$ is finite. Apply Thm.\ \ref{converse conormal} with $\mathcal{A}:=\mathcal{R}(\mathcal{M})$ and $\mathcal{B}:=\mathcal{R}(\mathcal{N})$. Note that by Prp.\ \ref{structure morphism} \rm{(2)} we have $\dim B_i = \dim X_i,x+e-1$. Then for each $i$ and each component $D_j$ of $b_{i}^{-1}x$ we must have $\codim (D_j, B_i) \geq 2$. But
$$\dim D_j \leq \dim B_i - \codim (D_j,B_i).$$
Hence  $\dim D_j \leq \dim X_i,x+e-3$ which contradicts with the assumption that $x \in S_{i}(\dim X_i,x+e-2)$ for some $i$. Thus, by Cor.\ \ref{ass. int. asym.}, $x \in \mathrm{Ass}_{X}(\mathcal{N}^{n}/\overline{\mathcal{M}^{n}})$ for $n$ large enough.
\end{proof}

It's worth mentioning that Thm.\ \ref{compl. desc.} recovers Burch's result  when $\mathcal{M}:=I$ is an ideal in $\mathcal{N}:=R$ whereas those of Rees, Katz and Rice do not. The reason is that our rank hypothesis allows for the possibility of $I$ to be contained in some of the minimal primes of $R$ but not in others. The fixed rank hypothesis imposed by Rees, Katz and Rice, however, implies that if $I$ does not consist entirely of nilpotents, then $I$ must avoid all minimal primes of $R$.  

The following corollary characterizes the embeddings of $\mathcal{M}$ into modules $\mathcal{N}$ of deficient analytic spread for which the points appearing in $\mathrm{Ass}_{X}(\mathcal{N}^{n}/\overline{\mathcal{M}^{n}})$ are pricesely the generic points of the irreducible components of $S_{i}(k)$. 
\begin{corollary}
Assume $\mathcal{M}$ is contained in  module $\mathcal{N}$ that has deficient analytic spread at each $x \in X$ with $\dim  X,x \geq 2$, such that the nonfree locus of
$\mathcal{N}/\mathcal{M}$ does not have codimension-one components. If $x \in \mathrm{Ass}_{X}(\mathcal{N}^{n}/\overline{\mathcal{M}^{n}})$ for some $n$ and $X$ is universally catenary, then $x$ is the generic point of an irreducible component of $S_{i}(\dim X_i,x+e-2)$ for some $i$. Conversely, if $x$ is the generic point of an irreducible component of $S_{i}(\dim X_i,x+e-2)$ for some $i$, then $x \in \mathrm{Ass}_{X}(\mathcal{N}^{n}/\overline{\mathcal{M}^{n}})$ for each $n$ large enough.
\end{corollary}
\begin{proof}
The forward direction of the statement follows from Thm.\ \ref{ass points-int.cl.}. The reverse direction follows from Thm.\ \ref{compl. desc.} noting that the hypothesis on the nonfree locus of $\mathcal{N}/\mathcal{M}$ implies that $\dim X,x \geq 2$.
\end{proof}
Assume $\mathcal{M}$ is free of rank $e$ at the generic point of each irreducible component of $X$. Suppose $\mathcal{M}$ is embedded in a free $R$-module $\mathcal{F}$ of rank $p$. As usual denote by $Z$ the nonfree locus of $\mathcal{F}/\mathcal{M}$. Denote by $\mathcal{G}$ the set of the generic points of the irreducible components of $X$. In general, if $p>e$, it's entirely possible that the set $\mathrm{Ass}_{X}(\mathcal{F}^{n}/\overline{\mathcal{M}^{n}})$ is equal to $\mathcal{G}$ as demonstrated by the  next proposition.

\begin{proposition}\label{direct summand} Assume $X$ is a local universally catenary scheme with closed point $x_0$. Assume  $\mathcal{M}$ is free of rank $e$ at each point $x \neq x_0$ with $p>e$. Furthermore, suppose $Z$ does not contain points $x$ with $\dim X_i,x=1$ for $X_i$ containing $x$. Finally, assume $\mathcal{M}$ has deficient analytic spread. 
Then  
$\mathcal{G}=\mathrm{Ass}_{X}(\mathcal{F}^{n}/\overline{\mathcal{M}^{n}})$ for each $n$.
\end{proposition}
\begin{proof}
Because $p>e$ by Thm.\ \ref{ass points-int.cl.} \rm{(1)} for each $i$ we have $\eta_i \in \mathrm{Ass}_{X}(\mathcal{F}^{n}/\overline{\mathcal{M}^{n}})$ for all $n$.

Suppose $x \neq \eta_i$ and $x \in \mathrm{Ass}_{X}(\mathcal{F}^{n}/\overline{\mathcal{M}^{n}})$ for some $n$. Note that by Thm.\ \ref{conormal} $x \neq x_0$ because $\codim D_j \geq 2$. 
By Thm.\ \ref{ass points-int.cl.} \rm{(2)} $x \in Z$. Because $Z$ does not contain points $y$ such that $\dim X_i,y=1$, then $\dim X_i,x \geq 2$ for each $X_i$ that contains $x$. By assumption $\mathcal{M}_x$ is free of rank $e$. Hence $\dim b_{i}^{-1}x = e-1$, whereas $\dim X_i,x+e-2 \geq e$. Therefore, $x \not \in S_{i}(\dim X_i,x+e-2)$ for each $X_i$ that contains $x$. Hence by Thm.\ \ref{ass points-int.cl.} $\rm{(3)}$ $x \not \in \mathrm{Ass}_{X}(\mathcal{F}^{n}/\overline{\mathcal{M}^{n}})$ for each $n$ which is contradiction. The proof is now complete.
\end{proof}

Nontrivial examples of $\mathcal{M}$ and $\mathcal{F}$ that satisfy the assumptions of Prp.\ \ref{direct summand} appear quite naturally in geometry. Consider a smooth complex projective variety $Y \subset \mathbb{P}^{N}$ with deficient dual, i.e. dual that has smaller dimension than the expected one, or in this case smaller than $N-1$ (cf.\ \cite{Ein1} and \cite{Ein2} for classification results and examples). Consider the affine cone over it. Denote by $X$ the germ of the cone at its vertex and denote the vertex by $x_0$. Since the base of $X$ is deficient, $X$ is not a complete intersection. Let $f_1, \ldots, f_p$ be the defining equations of $X$ in $\mathbb{C}^{N+1}$. Let $\mathcal{M}$ be the Jacobian module of $X$ - that is the $\mathcal{O}_{X,x_0}$-module spanned by the  columns of the  Jacobian matrix. It sits inside a free module $\mathcal{F}$ or rank $p$. Because $X$ has an isolated singularity, then $\mathcal{M}$ is a direct summand of $\mathcal{F}$ of rank $e$ locally off the vertex of the cone. Because $X$ is not a complete intersection we have $p>e$. Finally, because the base of the cone has deficient dual, then the conormal space of $X$, which is $\mathrm{Proj}(\mathcal{M})$, has fiber over the vertex of dimension less than expected (c.f. Sct.\ 1.5 in \cite{KT}). This bound in our setup translates to $\dim b^{-1}x_0 < \dim X+e-2$.

\end{document}